\documentclass[10pt]{article}

\usepackage{amssymb,amsthm,amsmath,hyperref}

\usepackage{graphicx,float}

\numberwithin{equation}{section}

\newtheorem{thm}{Theorem}[section]
\newtheorem{lem}{Lemma}[section]

\newtheorem{prop}{Proposition}[section]
\theoremstyle{definition}
\newtheorem{defn}{Definition}[section]
\theoremstyle{remark}
\newtheorem{rem}{Remark}[section]

\allowdisplaybreaks

\begin{document}

\title{Global well-posedness for the incompressible viscoelastic fluids in the critical $L^p$ framework}

\author{Ting Zhang\thanks{E-mail: zhangting79@zju.edu.cn}, Daoyuan Fang\thanks{E-mail: dyf@zju.edu.cn}
 \\
\textit{\small Department of Mathematics, Zhejiang University,
Hangzhou 310027, China} }
\date{}
\maketitle

\begin{abstract}
We investigate   global strong solutions for the
incompressible viscoelastic system of Oldroyd--B type with the initial data close to a stable equilibrium. We obtain the
existence and uniqueness of the global solution in a functional setting
invariant by the scaling of the associated equations, where the initial velocity has the same critical regularity
index as for the incompressible Navier--Stokes equations, and one
more derivative is needed for the deformation tensor.
Like the classical incompressible Navier-Stokes, one may construct the unique global   solution for a class of large highly oscillating initial velocity.  Our result also
implies that the deformation tensor $F$ has the same regularity as
the density of the compressible Navier--Stokes equations.
\end{abstract}

\section{Introduction}
In this paper, we consider the following system describing
incompressible viscoelastic fluids.
\begin{equation}
  \left\{\begin{array}{l}
    \nabla \cdot v=0, \ \ x\in \mathbb{R}^N,\ N\geq 2,\\
        v_t+v\cdot \nabla v+\nabla p=\mu\Delta v+\nabla\cdot\left[
        \frac{\partial W(F)}{\partial F}F^\top
        \right],\\
            F_t+v\cdot \nabla F=\nabla vF,\\
             F(0,x)=I+E_0(x),
   \ v(0,x)=v_0(x).
  \end{array}
  \right.\label{vis2-E1.1}
\end{equation}
Here, $v$, $p$, $\mu>0$, $F$ and $W(F)$ denote, respectively, the
velocity field of materials, pressure, viscosity, deformation tensor
and elastic energy functional.  The third equation is simply the
consequence of the chain law. It can also be regarded as the
consistence condition of the flow trajectories obtained from the
velocity field $v$ and also of those obtained from the deformation
tensor $F$ \cite{Dafermos,Gurtin,Lei,Lin,Liu}. Moreover, on the
right-hand side of the momentum equation, $\frac{\partial
W(F)}{\partial F}$ is the Piola--Kirchhoff stress tensor and
$\frac{\partial W(F)}{\partial F}F^\top$ is the Cauchy--Green
tensor. The latter is the change variable (from Lagrangian to
Eulerian coordinates) form of the former one \cite{Lei}. The above
system is equivalent to the usual Oldroyd--B model for viscoelastic
fluids in the case of infinite Weissenberg number \cite{Larson}. On
the other hand, without the viscosity term, it represents exactly
the incompressible elasticity in Eulerian coordinates. We refer to
\cite{Byron,Dafermos,Larson,Lin2,Lions,Liu} and their references for
the detailed derivation and physical background of the above system.

Throughout this paper, we will use the notations of
    $$
    (\nabla v)_{ij}=\frac{\partial v_i}{\partial x_j},
     \ (\nabla v F)_{ij}=(\nabla v)_{ik}F_{kj},
     \ (\nabla\cdot F)_{i}=\partial_j F_{ij},
    $$
and summation over repeated indices will always be understood.

For incompressible viscoelastic fluids, Lin et al. \cite{Lin} proved
the global existence of classical small solutions for the
two--dimensional case with the initial data $v_0,E_0=F_0-I\in
H^k(\mathbb{R}^2)$, $k\geq2$, by introducing an auxiliary vector
field to replace the transport variable $F$. Using the method in
\cite{Kawashima} for the damped wave equation, they \cite{Lin} also
obtained the global existence of classical small solutions for the
three--dimensional case with the initial data $v_0,E_0\in
H^k(\mathbb{R}^3)$, $k\geq3$. Lei and Zhou \cite{Lei3} obtained the
same results for the two--dimensional case with the initial data
$v_0,E_0\in H^s(\mathbb{R}^2)$, $s\geq4$. via the incompressible
limit working directly on the deformation tensor $F$. Then, Lei et
al. \cite{Lei2} proved the global existence for two--dimensional
small-strain viscoelasticity with $H^2(\mathbb{R}^2)$ initial data
and without assumptions on the smallness of the rotational part of
the initial deformation tensor.  It is worth noticing that the
global existence and uniqueness for the large solution of the
two--dimensional problem is still open.  Recently, by introducing an
auxiliary function $w=\Delta v+\frac{1}{\mu}\nabla\cdot E$, Lei et
al. \cite{Lei} obtained a weak dissipation on the deformation $F$
and the global existence of classical small solutions to
$N$--dimensional system with the initial data $v_0,E_0\in
H^2(\mathbb{R}^N)$ and $N=2,3$. All these results need that the
initial data $v_0$ and $E_0$ have the same regularity, and the
regularity index of the initial velocity $v_0$ is bigger than the
critical regularity index for the classical incompressible
Navier--Stokes equations.

There are some classical results for the following
incompressible Navier--Stokes equations.
\begin{equation}
  \left\{\begin{array}{l}
        v_t+v\cdot \nabla v+\nabla p=\mu\Delta v,\\\nabla \cdot v=0, \\
   \ v(0,x)=v_0(x).
  \end{array}
  \right.\label{vis2-E1.1-NS}
\end{equation}
In 1934, J. Leray proved the existence of global weak solutions for
(\ref{vis2-E1.1-NS}) with divergence--free $u_0\in L^2$ (see
\cite{Leray}). Then, H. Fujita and T. Kato (see \cite{Fujita})
obtain the uniqueness with $u_0\in \dot{H}^{\frac{N}{2}-1}$. The
index $s= N/2-1$ is critical for (\ref{vis2-E1.1-NS}) with initial
data in $\dot{H}^s$: this is the lowest index for which uniqueness
has been proved (in the framework of Sobolev spaces). This fact is
closely linked to the concept of scaling invariant space. Let us
precise what we mean. For all $l > 0$, system (\ref{vis2-E1.1-NS})
is obviously invariant by the transformation
$$u(t, x) \rightarrow u_l(t, x) := lu(l^2t, lx),
u_0(x)\rightarrow u_{0l}(x):= lu_0(lx),$$
 and a straightforward
computation shows that $\|u_0\|_{\dot{H}^{\frac{N}{2}-1}}
=\|u_{0l}\|_{\dot{H}^{\frac{N}{2}-1}}$. This idea of using a
functional setting invariant by the scaling of (\ref{vis2-E1.1-NS})
is now classical and originated many works.   In \cite{Cannone93},     M.
Cannone, Y. Meyer and F. Planchon proved that: if the initial data
satisfy
    $$
    \|v_0\|_{\dot{B}^{-1+\frac{N}{p}}_{p,\infty}}\leq c\nu,
    $$
for  $p>N$ and some constant $c$ small enough, then the classical
Navier-Stokes system (\textit{NS}) is globally wellposed. Refer to
\cite{Cannone,Chemin,Chemin07-0,Koch01} for a recent panorama.

In this paper, using the method in \cite{Danchin2010,Chen}
studying the compressible Navier--Stokes system, we are concerned with
the existence and uniqueness of a solution for the initial data in a
functional space with minimal regularity order $v_0\in\dot{B}^{-1+\frac{N}{p}}_{p,1}$. Although
 the equation (\ref{vis2-E1.1})$_3$ for the deformation tensor $F$ is identical to the equation for
the vorticity $\omega=\nabla\times v$ of the Euler equations
    $$
    \partial_t \omega+  v\cdot \nabla \omega=\nabla v\omega,
    $$
our result implies that the deformation tensor $F$ has the same
regularity as the density of the compressible Navier--Stokes
equations.

 At this stage, we will use scaling considerations for
 (\ref{vis2-E1.1}) to guess which spaces may be critical. We observe
  that  (\ref{vis2-E1.1}) is invariant by the transformation
    $$
    (v_0(x), F_0(x))\rightarrow (lv_0(lx),F_0(lx)),
    $$
        $$
(v(t,x), F(t,x),P(t,x))\rightarrow
(lv(l^2t,lx),F(l^2t,lx),l^2P(l^2t,lx)),
        $$
up to a change of the elastic energy functional $W$ into $l^2W$.

\begin{defn}\label{vis2-D1.1}
  A functional space $E\subset (\mathcal{S}'(\mathbb{R}^N))^N\times (
  \mathcal{S}'(\mathbb{R}^N)^{N\times N}$ is called
a critical space if the associated norm is invariant under the
transformation $
    (v(x), F(x))\rightarrow (lv(lx),F(lx))
    $ (up to a constant
independent of $l$).
\end{defn}

In this paper,
$(\dot{B}^{\frac{N}{p}-1}_{p,1})^N\times(\dot{B}^{\frac{N}{p}}_{p,1})^{N\times N}$, $p\in[2,\infty)$, is a
critical space.

Assume that $E_0:=F_0-I$ and $v_0$ satisfy the
following constraints:
\begin{equation}
    \nabla \cdot v_0=0,
     \  \mathrm{det}(I+E_0)=1,
   \    \nabla\cdot E^\top_0=0,\label{vis2-E1.3-0}
       \end{equation}
and
\begin{equation}
       \partial_m E_{0ij}-\partial_jE_{0im}
       =E_{0lj}\partial_lE_{0im}-E_{0lm}\partial_lE_{0ij}.
  \label{vis2-E1.3}
\end{equation}
The first three of these expressions are just the consequences of
the incompressibility condition  and the last one can be understood
as the consistency condition for changing variables between the
Lagrangian and Eulerian coordinates \cite{Lei}.

For simplicity, we only consider the case of Hookean elastic
materials: $W(F)=|F^2|$. Define the usual strain tensor by the form
    \begin{equation}
      E=F-I.
    \end{equation}
Then, the system (\ref{vis2-E1.1}) is
\begin{equation}
  \left\{\begin{array}{l}
    \nabla \cdot v=0, \ \ x\in \mathbb{R}^N,\ N\geq 2,\\
        v_{it}+v\cdot \nabla v_i+\partial_i p=\mu\Delta v_i+ E_{jk}\partial_j E_{ik}+\partial_jE_{ij},\\
            E_t+v\cdot \nabla E=\nabla vE+\nabla v,\\
                (v,E)(0,x)=(v_0,E_0)(x).
  \end{array}
  \right.\label{vis2-E1.4}
\end{equation}

The global well-poesdness result had been obtained with the small initial data
     $v_0\in   (\dot{B}_{2,1}^{\frac{N}{2}-1})^N$ and $E_0\in (\dot{B}_{2,1}^\frac{N}{2}\cap \dot{B}_{2,1}^{\frac{N}{2}-1})^{N\times N}$ independently in \cite{Qian,zhang10}.
\begin{thm}[\cite{Qian,zhang10}]\label{vis2-T1.1}
    Suppose that the
   initial data satisfies the incompressible constraints (\ref{vis2-E1.3-0}),  $v_0\in
    (\dot{B}_{2,1}^{\frac{N}{2}-1})^N$ and
 $E_0\in (\dot{B}_{2,1}^\frac{N}{2})^{N\times N}$.
    Then there exist $T>0$ and a unique local solution for system (\ref{vis2-E1.4}) that satisfies
        $$
  (v,E)\in  U^\frac{N}{2}_T,
        $$
    \begin{equation}
      \|(v,E)\|_{U^\frac{N}{2}_T}\leq C(\|E_0\|_{
      \dot{B}_{2,1}^\frac{N}{2}}+\|v_0\|_{\dot{B}_{2,1}^{\frac{N}{2}-1}}),
    \end{equation}
  and
    \begin{equation}
       \nabla \cdot v =0,
     \  \mathrm{det}(I+E )=1,
   \    \nabla\cdot E^\top =0,
    \end{equation}
where
$$
    U^s_T=\left(L^1([0,T]; \dot{B}_{2,1}^{s+1})\cap C([0,T];   \dot{B}_{2,1}^{s-1})
    \right)^N\times \left( C([0,T]; \dot{B}_{2,1}^s)
    \right)^{N\times N},$$

 Furthermore,  suppose that the
   initial data satisfies the incompressible constraints
   (\ref{vis2-E1.3-0})--(\ref{vis2-E1.3}), $v_0\in   (\dot{B}_{2,1}^{\frac{N}{2}-1})^N$,
    $E_0\in (\dot{B}_{2,1}^\frac{N}{2}\cap \dot{B}_{2,1}^{\frac{N}{2}-1})^{N\times N}$
     and
        \begin{equation}
          \|E_0\|_{\dot{B}_{2,1}^\frac{N}{2}\cap
          \dot{B}_{2,1}^{\frac{N}{2}-1}}+\|v_0\|_{\dot{B}_{2,1}^{\frac{N}{2}-1}}\leq \lambda,
        \end{equation}
        where $\lambda$ is a small positive constant.
Then there
  exists a unique global   solution for system (\ref{vis2-E1.4}) that satisfies
    $$
(v,E)\in{V^\frac{N}{2}}
    $$
    \begin{equation}
      \|(v,E)\|_{V^\frac{N}{2}}\leq C(\|E_0\|_{B^\frac{N}{2}\cap
      \dot{B}_{2,1}^{\frac{N}{2}-1}}+\|v_0\|_{\dot{B}_{2,1}^{\frac{N}{2}-1}}),
    \end{equation}
    where
$$
    V^s= \left(L^1(\mathbb{R}^+; \dot{B}_{2,1}^{s+1})\cap C(\mathbb{R}^+;   \dot{B}_{2,1}^{s-1})
    \right)^N\times \left(L^2(\mathbb{R}^+; \dot{B}_{2,1}^s)\cap C(\mathbb{R}^+;
    \dot{B}_{2,1}^s\cap \dot{B}_{2,1}^{s-1})
    \right)^{N\times N}.$$
\end{thm}

In this paper, we will obtain the following theorem, where
    $$
    f^l=\sum_{2^k\leq R_0} \Delta_k f,
        f^h=\sum_{2^k> R_0} \Delta_k f,
    $$
    and $R_0$ is as in Proposition \ref{vis2-P4.2}.
\begin{thm}\label{vis2-T1.2}
Suppose that the
   initial data satisfies the incompressible constraints
   (\ref{vis2-E1.3-0})--(\ref{vis2-E1.3}),
    $E_0\in \dot{B}_{p_1,1}^\frac{N}{p_1}$, $E_0^l\in\dot{B}_{2,r}^{s}$,
    $E_0^h\in\dot{B}_{2,r}^{s+1}$,
    $v_0\in   \dot{B}_{p_1,1}^{\frac{N}{p_1}-1}\cap \dot{B}^s_{2,r}$,
 for some $p_1\in [2,\infty)$, $p_2\in[p_1,\infty)$, $r\in[1,\infty]$ and
 $s\in\mathbb{R}$ such that
    \begin{equation}
      s\in\left\{
      \begin{array}{ll}
    (-1,\frac{N}{2}-1)&\textrm{ if }r>1,\\
        (-1,\frac{N}{2}-1]&\textrm{ if }r=1.
      \end{array}
      \right.
    \end{equation}
There exists a constant  $\lambda$   depending only on $N$, $p_2$
and $s$ such that if
        \begin{equation}
  \|E_0^l\|_{\dot{B}_{2,r}^s}+
  \|E_0^h\|_{\dot{B}_{2,r}^{s+1}\cap \dot{B}_{p_2,1}^{\frac{N}{p_2}}}
  +\|v_0\|_{\dot{B}_{2,r}^s}+\|v_0^h\|_{\dot{B}_{p_2,1}^{\frac{N}{p_2}-1}}
  \leq \lambda,
        \end{equation}
then there
  exists a global solution for system (\ref{vis2-E1.4}) that satisfies
    $$
(v,E)\in{V^s_{p_1,r}}
    $$
    where
\begin{eqnarray*}
    V^s_{p,r}&=&
   \left\{ v\in \left(\widetilde{C}(\mathbb{R}^+;   \dot{B}_{2,r}^{s}\cap  \dot{B}_{p,1}^{\frac{N}{p}-1})
    \cap \widetilde{L}^1(\mathbb{R}^+; \dot{B}_{2,r}^{s+2}\cap \dot{B}_{p,1}^{\frac{N}{p}+1})
    \right)^N\right.\\
 &&    E\in\left(\widetilde{C}(\mathbb{R}^+; \dot{B}_{2,r}^{s+1} \cap
   \dot{B}_{p,1}^\frac{N}{p})
 \cap L^2(\mathbb{R}^+; \dot{B}^\frac{N}{p}_{p,1})\right)^{N\times N},
   \\
 && \left.
E^l\in \left(\widetilde{L}^2(\mathbb{R}^+; \dot{B}_{2,r}^{s+1})\cap
  \widetilde{L}^1(\mathbb{R}^+; \dot{B}^{s+2}_{2,r})\right)^{N\times N}, E^h\in \left(
  \widetilde{L}^1(\mathbb{R}^+; \dot{B}^{s+1}_{2,r})
    \right)^{N\times N}
    \right\}.
    \end{eqnarray*}
 Furthermore, the solution is unique in $ V^s_{p_1,r}$ if $p_1\leq
 2N$.
\end{thm}
\begin{rem}
This result allows us to construct the unique global solution for the high oscillating initial velocity $v_0$. For example,
    $$
    v_0(x)=\varepsilon^{\frac{N}{p}-1}\sin(\frac{x_1}{\varepsilon})\phi(x), \ \phi(x)\in \mathcal{S}(\mathbb{R}^N),\ p\in[N,\infty),
    $$
which satisfies
    $$
    \|v_0\|_{\dot{B}^s_{2,r}}+\|v_0^h\|_{\dot{B}^{\frac{N}{p_2}-1}_{p_2,1}}\ll1,
    $$
    if  $s<-1+\frac{N}{p}$, $p_2>p$, $\varepsilon$ is small enough, (Proposition 2.9, \cite{Chen}).
\end{rem}
\begin{rem}
  The $L^2$-decay in time for $E$ is a key point in the proof of
the global existence. We shall also get a $L^1-$decay in a space $E^l\in \left(
  \widetilde{L}^1(\mathbb{R}^+; \dot{B}^{s+2}_{2,r})\right)^{N\times N}, E^h\in \left(
  \widetilde{L}^1(\mathbb{R}^+; \dot{B}^{s+1}_{2,r})
    \right)^{N\times N}$.
\end{rem}

\begin{rem}\label{vis2-r1.2}
  Theorem \ref{vis2-T1.2} implies that the deformation tensor $F$ has similar
  property as the density of the compressible Navier--Stokes system
  \cite{Danchin2010,Chen}. And we think that the
incompressible viscoelastic system is similar to the  compressible
Navier--Stokes system.
\end{rem}

\begin{rem}
Similar to the  compressible
  Navier--Stokes system \cite{Danchin2010,Chen},  the initial data do not
really belong to a critical space in the sense of Definition
\ref{vis2-D1.1}. We indeed made the additional assumptions $E_0^l\in
\dot{B}^{s}_{2,r}$, $E^h_0\in \dot{B}^{s+1}_{2,r}$, $v_0\in \dot{B}^{s}_{2,r}$. On the
other hand, our scaling considerations do not take care of the
  Cauchy--Green tensor term. A careful study of the linearized system (see
 Proposition
\ref{vis2-P4.2} below) besides indicates that such an assumption may
be unavoidable.
\end{rem}
\begin{rem}
Considering the general viscoelastic model (\ref{vis2-E1.1}), if the
strain energy function satisfies the strong Legendre--Hadamard
ellipticity condition
    \begin{equation}
    \frac{\partial^2W(I)}{\partial F_{il}\partial F_{jm}}=(\alpha^2-2\beta^2)\delta_{il}\delta_{jm}+\beta^2(\delta_{im}\delta_{jl}
    +\delta_{ij}\delta_{lm}),
     \textrm{ with }\alpha>\beta>0,
    \end{equation}
 and the reference configuration stress--free condition
    \begin{equation}
    \frac{\partial W(I)}{\partial F}=0,
    \end{equation}
then we can obtain the same results as that in Theorem \ref{vis2-T1.2}.
\end{rem}

In this paper, we introduce the following function:
    $$
    c=\Lambda^{-1}{ \nabla\cdot E},
    $$ where $\Lambda^{s} f=  \mathcal{F}^{-1} (|\xi|^s \hat{f})$.
Then, the system (\ref{vis2-E1.4}) reads
    \begin{equation}
  \left\{\begin{array}{l}
    \nabla \cdot v=\nabla \cdot c=0, \ \ x\in \mathbb{R}^N,\ N\geq2,\\
        v_{it}+v\cdot \nabla v_i+\partial_i p=\mu\Delta v_i+ E_{jk}\partial_j E_{ik}+\Lambda c_i,\\
            c_t+v\cdot \nabla c +[\Lambda^{-1}\nabla\cdot, v\cdot ]\nabla E=\Lambda^{-1} \nabla\cdot(\nabla vE)-\Lambda v,\\
                     \Delta E_{ij}= \Lambda\partial_j c_i+\partial_k(\partial_k E_{ij}-\partial_jE_{ik }),\\
                (v,c)(0,x)=(v_0,\Lambda^{-1}\nabla\cdot E_0)(x).
  \end{array}
  \right.\label{vis2-E1.14}
\end{equation}
So, we need to study the following mixed parabolic--hyperbolic
linear system with a convection term:
    \begin{equation}
  \left\{\begin{array}{l}
    \nabla \cdot v=\nabla \cdot c=0, \ \ x\in \mathbb{R}^N,\ N\geq2,\\
                               v_{it}+u\cdot \nabla v_i+\partial_i p-\mu\Delta v_i-\Lambda c_i=G,\\
     c_t+u\cdot \nabla c+\Lambda v =L,\\
                (v,c)(0,x)=(v_0,c_0)(x),
  \end{array}
  \right.\label{vis2-E1.15}
\end{equation}
where div$u=0$ and $c_0=\Lambda^{-1}\nabla\cdot E_0$. This system is
similar to the system
    \begin{equation}
      \left\{
      \begin{array}{l}
    c_t+v\cdot\nabla c+\Lambda d=F,\\
        d_t+v\cdot\nabla d-\bar{\mu}\Delta d-\Lambda c=G,
      \end{array}
      \right.
    \end{equation}
for compressible Navier-Stokes system  \cite{Danchin2010,Chen}. Using the similar
method studying the compressible Navier--Stokes system  in
 \cite{Danchin2010,Chen}, we can obtain some important estimates for the system (\ref{vis2-E1.15}).

As for the related studies on the existence of solutions to
nonlinear elastic systems, there are works by Sideris \cite{Sideris}
and Agemi \cite{Agemi} on the global existence of classical small
solutions to three--dimensional compressible elasticity, under the
assumption that the nonlinear terms satisfy the null conditions. The
global existence for three--dimensional incompressible elasticity
was then proved via the incompressible limit method in
\cite{Sideris2004} and  by a different method in \cite{Sideris2007}. It
is worth noticing that the  global existence and uniqueness for the
corresponding two--dimensional problem is still open.

As for the density-dependent incompressible viscoelastic fluids, Hu and Wang \cite{Hu}
obtained the existence and uniqueness of the global strong solution with small initial
data in $\mathbb{R}^3$.

The rest of this paper is organized as follows. In Section
\ref{vis2-S2}, we state  three lemmas describing the intrinsic
properties of viscoelastic system. In Section \ref{vis2-S3}, we
present the functional tool box: Littlewood--Paley decomposition,
product laws in Sobolev and hybrid Besov spaces. The next section is
devoted to the study of some linear models associated to
(\ref{vis2-E1.4}). In Section \ref{vis2-s5}, we give some a priori estimates.
 At last,
we will study  the
global well-posedness for (\ref{vis2-E1.4}).

\section{Basic mechanics of viscoelasticity}\label{vis2-S2}
Using the similar arguments as that in the proof of Lemmas 1--3 in
\cite{Lei}, we can easily obtain the following three lemmas.

\begin{lem}
  Assume that $\mathrm{det}(I+E_0)=1$ is satisfied and $(v,F)$ is the solution of system (\ref{vis2-E1.4}). Then the following is always true:
  \begin{equation}
    \mathrm{det}(I+E)=1,
  \end{equation}
for all time $t\geq0$, where the usual strain tensor $E=F-I$.
\end{lem}

\begin{lem}\label{vis2-L2.2}
  Assume that $\nabla\cdot E^\top_0=0$ is satisfied, then the solution $(v,F)$  of system (\ref{vis2-E1.4}) satisfies the following identities:
  \begin{equation}
    \nabla \cdot F^\top=0, \ \textrm{ and }\ \nabla\cdot E^\top=0,
  \end{equation}
for all time $t\geq0$.
\end{lem}

\begin{lem}\label{vis2-L2.3}
  Assume that  (\ref{vis2-E1.3}) is satisfied and $(v,F)$ is the solution of system (\ref{vis2-E1.4}). Then the following is always true:
  \begin{equation}
    \partial_m E_{ij}-\partial_jE_{im}
       =E_{lj}\partial_lE_{im}-E_{lm}\partial_lE_{ij},
  \end{equation}
for all time $t\geq0$.
\end{lem}

\section{Littlewood--Paley theory and Besov spaces}\label{vis2-S3}
The proof of most of the results presented in this paper requires a
dyadic decomposition of Fourier variable (\textit{Littlewood--Paley
composition}). Let us briefly explain how it may be built in the
case $x\in \mathbb{R}^N$, $N\geq2$, (see
\cite{Danchin,Danchin2001,Danchin2003}).

 Let $\mathcal{S}(\mathbb{R}^N)$ be the Schwarz class. $\varphi(\xi)$ is a  smooth function valued in [0,1]
such that
    $$
    \textrm{supp}\varphi\subset\{\frac{3}{4}\leq|\xi|\leq\frac{8}{3}\}
    \ \textrm{ and }
     \sum_{q\in\mathbb{Z}}\varphi(2^{-q}\xi)=1,\  |\xi|\not=0.
    $$
Let $h(x)=(\mathcal{F}^{-1}\varphi)(x)$. For $f\in\mathcal{S'}$
(denote the set of temperate distributes, which is the dual one of
$\mathcal{S}$), we can define the homogeneous dyadic blocks as
follows:
        $$
        \Delta_{q}f(x):=\varphi(2^{-q}D)f(x)=2^{Nq}\int_{\mathbb{R}^N} h(2^{q}y)f(x-y)dy,
        \ \mbox{if}\        q\in\mathbb{Z},
        $$
where    $\mathcal{F}^{-1} $ represents the    inverse Fourier
transform. Define the low frequency cut-off by
    $$
    S_{q}f(x):=\sum_{p\leq
    q-1}\Delta_{p}f(x)=\chi(2^{-q}D)f(x).
    $$
 The  Littlewood--Paley
decomposition has nice properties of quasi-orthogonality,
    $$
    \Delta_{p}\Delta_{q}f_1\equiv 0,\ \mbox{if}\ |p-q|\geq 2,
    $$
and
        $$
        \Delta_{q}(S_{p-1}f_1\Delta_{p}f_2)\equiv 0,
               \ \mbox{if}\ |p-q|\geq 5.
        $$

\begin{lem}[Bernstein]\label{vis2-L3.1}
Let $k\in\mathbb{N}$ and $0<R_{1}<R_{2}$. There exists a constant
$C$ depending only on $R_{1},R_{2}$ and $N$ such that for all
$1\leqslant a\leqslant b\leqslant\infty$ and $f\in L^{a}$, we have
$$\mathrm{Supp}\ \mathcal{ F}f\subset \textsl{B}(0,R_{1}\lambda)\Rightarrow\sup_{|\alpha|=k}\|\partial^{\alpha}f\|_{L^{b}}\leqslant C^{k+1}\lambda^{k+N(\frac{1}{a}-\frac{1}{b})}\|f\|_{L^{a}};$$
$$\mathrm{Supp}\ \mathcal{ F}f\subset \textsl{C}(0,R_{1}\lambda,R_{2}\lambda)\Rightarrow C^{-k-1}\lambda^{k}\|f\|_{L^{a}}\leqslant \sup_{|\alpha|=k}\|\partial^{\alpha}f\|_{L^{a}}\leqslant C^{k+1}\lambda^{k}\|f\|_{L^{a}}.$$
Here, $\mathcal{F}$ represents the Fourier transform.
\end{lem}

The Besov space can be characterized in virtue of the
Littlewood--Paley decomposition.
\begin{defn}
Let $1\leq p\leq\infty$ and $s\in \mathbb{R}$. For $1\leq r\leq
\infty$, the Besov spaces $\dot{B}^{s}_{p,r}(\mathbb{R}^N)$,
$N\geq2$, are defined by
    $$
    f\in \dot{B}^{s}_{p,r}(\mathbb{R}^N)
    \Leftrightarrow \left\| 2^{qs}\|\Delta_{q}f\|_{L^{p}(\mathbb{R}^N)}
    \right\|_{l^r_q}<\infty
    $$
and $B^{s}(\mathbb{R}^N)=\dot{B}^{s}_{2,1}(\mathbb{R}^N)$.
\end{defn}

The definition of $\dot{B}^s_{p,r}(\mathbb{R}^N)$ does not depend on
the choice of the Littlewood--Paley decomposition.  Let us recall
some classical estimates in Sobolev spaces for the product of two
functions \cite{Danchin2010,Danchin2005}.

Formally, Bony's decomposition is defined by
    $$
    uv=T_uv+T_vu+R(u,v)=T_uv+T'_vu,
    $$
with
    $$
    T_uv=\sum_{q\in\mathbb{Z}}S_{q-1}u\Delta_qv,
    \ R(u,v)=\sum_{q\in\mathbb{Z}}\Delta_qu\widetilde{\Delta}_qv,
    \ \widetilde{\Delta}_qv=\sum_{|q'-q|\leq1}\Delta_{q'}v.
    $$
\begin{prop}\label{vis2-P3.1-0}
  For any $(s,p,r)\in \mathbb{R}\times[1,\infty]^2$ and $t<0$, there exists a constant $C$ such that
    $$
    \|T_uv\|_{\dot{B}^s_{p,r}}\leq C\|u\|_{L^\infty}\|v\|_{\dot{B}^s_{p,r}}
    \textrm{ and }
        \|T_uv\|_{\dot{B}^{s+t}_{p,r}}\leq C\|u\|_{\dot{B}^t_{p,r}}\|v\|_{\dot{B}^s_{p,r}}.
    $$
  For any $(s_1,p_1,r_1)$ and $(s_2,p_2,r_2)$ in $\mathbb{R}\times[1,\infty]^2$ and $t<0$, there exists a constant $C$ such that
   \begin{itemize}
     \item if $s_1+s_2>0$, $\frac{1}{p}=\frac{1}{p_1}+\frac{1}{p_2}\leq1$ and $\frac{1}{r}=\frac{1}{r_1}+\frac{1}{r_2}\leq1$, then
        $$
        \|R( u,v)\|_{\dot{B}^{s_1+s_2}_{p,r}}\leq C\|u\|_{\dot{B}^{s_1}_{p_1,r_1}}\|v\|_{\dot{B}^{s_2}_{p_2,r_2}};
    $$
     \item  if $s_1+s_2=0$, $\frac{1}{p}=\frac{1}{p_1}+\frac{1}{p_2}\leq1$ and $\frac{1}{r}=\frac{1}{r_1}+\frac{1}{r_2}\geq1$, then
        $$
        \|R( u,v)\|_{\dot{B}^{0}_{p,r}}\leq C\|u\|_{\dot{B}^{s_1}_{p_1,r_1}}\|v\|_{\dot{B}^{s_2}_{p_2,r_2}}.
    $$
   \end{itemize}
\end{prop}

\begin{prop}\label{vis2-P3.1}
  Let $1\leq r,p,p_1,p_2\leq\infty$. Then following inequalities hold true:
    \begin{equation}
    \|uv\|_{\dot{B}^{s}_{p,r}}\lesssim \|u\|_{L^\infty}\|v\|_{\dot{B}^{s }_{p,r}}
    +\|v\|_{L^\infty}\|u\|_{\dot{B}^{s }_{p,r}},\ \textrm{ if  }\ s>0,\label{vis2-E3.1}
    \end{equation}
        \begin{equation}
    \|uv\|_{\dot{B}^{s_1+s_2-\frac{N}{p}}_{p,r}}\lesssim \|u\|_{\dot{B}^{s_1}_{p,r}}\|v\|_{\dot{B}^{s_2}_{p,\infty}},
    \ \textrm{ if }\ s_1,s_2<\frac{N}{p}\ \textrm{ and }\ s_1+s_2>0,
        \end{equation}
        \begin{equation}
    \|uv\|_{\dot{B}^{s}_{p,r}}\lesssim \|u\|_{\dot{B}^{s}_{p,r}}\|v\|_{\dot{B}^{\frac{N}{p}}_{p,\infty}\cap L^\infty},
    \ \textrm{ if }\ |s| <\frac{N}{p},
        \end{equation}
        \begin{equation}
    \|uv\|_{\dot{B}^{-\frac{N}{p} }_{p,\infty}}\lesssim \|u\|_{\dot{B}^{s}_{p,1}}
    \|v\|_{\dot{B}^{-s}_{p,\infty} },
    \ \textrm{ if }\  s  \in(-\frac{N}{p},\frac{N}{p}],\ p\geq2.
        \end{equation}
\end{prop}

Finally, we state the definition of
$\widetilde{L}^q(0,T;B^{s}_{p,r})$, which is first introduced by
J.-Y. Chemin and N. Lerner in \cite{Chemin1992}.

\begin{defn}
  We denote by $\widetilde{L}^q(0,T;\dot{B}^{s}_{p,r})$  the space of distributions, which
  is the completion of $\mathcal{S}(\mathbb{R}^d)$ by the following
  norm:
    $$
        \|a\|_{\widetilde{L}^q(0,T;\dot{B}^{s}_{p,r})}= \left\|2^{sk}
        \|\Delta_k
        a\|_{L^q(0,T;L^p(\mathbb{R}^d))}\right\|_{l^r_k}.
    $$
\end{defn}
Similar results of Propositions \ref{vis2-P3.1-0}--\ref{vis2-P3.1} in space $\widetilde{L}^q(0,T;\dot{B}^{s}_{p,r})$  also hold  \cite{Danchin2010,Danchin2005}.

Denote
 \begin{equation}
   \|v\|_{\widetilde{L}^1_T(\tilde{B}^{s+2,s+1}_{2,r})}
   : =\left\|2^{p(s+1)}\min(R_0,2^p)\|\Delta_p v\|_{L^1_T(L^2_x)} \right\|_{l^r_p}
    \end{equation}

\begin{lem}\label{vis2-L3.2}
  Let $s>-1$ and $r\in[1,\infty]$, we have
  \begin{eqnarray}
    \|uv\|_{\widetilde{L}^1_T(\tilde{B}^{s+2,s+1}_{2,r})}
    \leq C\|u\|_{L^\infty_T(L^\infty_x)}\|v\|_{\widetilde{L}^1_T(\tilde{B}^{s+2,s+1}_{2,r})}
    +C\|v\|_{L^\infty_T(L^\infty_x)}
   \|u\|_{\widetilde{L}^1_T(\tilde{B}^{s+2,s+1}_{2,r})}.
  \end{eqnarray}
\end{lem}
\begin{proof}
We write, for $p\in\mathbb{Z}$,
$$
\Delta_p T_uv=\sum_{|q-p|\leq 3}\Delta_p(S_{q-1}u\Delta_q v),
$$
hence,
    $$
    \|\Delta_p T_uv\|_{L^1_T(L^2)}
    \leq C\sum_{|q-p|\leq 3}\|u\|_{L^\infty_T(L^\infty_x)}\|\Delta_q v\|_{L^1_T(L^2)},
    $$
and
 \begin{equation}
    \|T_uv\|_{\widetilde{L}^1_T(\tilde{B}^{s+2,s+1}_{2,r})}
    \leq C\|u\|_{L^\infty_T(L^\infty_x)}\|v\|_{\widetilde{L}^1_T(\tilde{B}^{s+2,s+1}_{2,r})}.
 \end{equation}
Similarly, we have
    \begin{equation}
    \|T_vu\|_{\widetilde{L}^1_T(\tilde{B}^{s+2,s+1}_{2,r})}
    \leq C\|v\|_{L^\infty_T(L^\infty_x)}\|u\|_{\widetilde{L}^1_T(\tilde{B}^{s+2,s+1}_{2,r})}.
 \end{equation}
 We write, for $p\in\mathbb{Z}$,
$$
\Delta_p R(u,v)=\sum_{ q\geq p-2}\Delta_p(\Delta_{q}u\widetilde{\Delta}_q v),
$$
hence,
\begin{eqnarray*}
&&2^{p(s+1)}\min(R_0,2^p)\|\Delta_p R(u,v)\|_{L^1_T(L^2)}\\
&\leq &
C\sum_{ q\geq p-2}
2^{(p-q)(s+1)}\frac{\min(R_0,2^p)}{\min(R_0,2^q)}
2^{q(s+1)}\min(R_0,2^q)\|\Delta_{q}u\|_{L^1_T(L^2)}
\| v\|_{L^\infty_T(L^\infty_x)}\\
&\leq &
C\sum_{ q\geq p-2}
2^{(p-q)(s+1)} c_q \| u\|_{\widetilde{L}^1_T(\tilde{B}^{s+2,s+1}_{2,r})}
\| v\|_{L^\infty_T(L^\infty_x)},
\end{eqnarray*}
and
        \begin{equation}
    \|R(u,v)\|_{\widetilde{L}^1_T(\tilde{B}^{s+2,s+1}_{2,r})}
    \leq C\|v\|_{L^\infty_T(L^\infty_x)}\|u\|_{\widetilde{L}^1_T(\tilde{B}^{s+2,s+1}_{2,r})},
 \end{equation}
 where we use the fact that $s>-1$.
\end{proof}
\section{A linear model with convection}\label{vis2-S4}
After the change of function
    $$
    c=\Lambda^{-1}{ \nabla\cdot E},
    $$
the system (\ref{vis2-E1.4}) reads (\ref{vis2-E1.14}). At first, we
will study the following mixed linear system
    \begin{equation}
  \left\{\begin{array}{l}
    \nabla \cdot v=\nabla \cdot c=0, \ \ x\in \mathbb{R}^N,\ N\geq2,\\
      v_{it}+T_u\cdot \nabla v_i+\partial_i p-\mu\Delta v_i-\Lambda c_i=G_i,\\
                     c_t+T_u\cdot \nabla c+\Lambda v =L,\\
                (v,c)(0,x)=(v_0,c_0)(x),
  \end{array}
  \right. \label{vis2-E4.1}
\end{equation}
where div$u=0$ and $c_0=\Lambda^{-1}\nabla\cdot E_0$.

 Using the
similar arguments as that in  \cite{Danchin} (Proposition 2.3), we
can obtain the following proposition and omit the details. The main
different is that there is a pressure term $\nabla p$ in
(\ref{vis2-E4.1})$_2$. Using that fact that $\nabla\cdot
v=\nabla\cdot c=0$, we have $\int v\cdot\nabla p = \int c\cdot\nabla
p =0$, and obtain the following proposition.
\begin{prop}\label{vis2-P4.3}
  Let $(v,c)$ be a solution of (\ref{vis2-E4.1}) on $[0,T]$, $\rho\in \mathbb{R}$
  and
  $\widetilde{U}(t)=\int^t_0\|\nabla u(\tau)\|_{L^{\infty}}d\tau$.
   The following estimate holds:
    \begin{eqnarray*}
     &&\|(v,c)^l \|_{\widetilde{L}^\infty_t(\dot{B}^\rho_{2,r})\cap
     \widetilde{L}^1_t(\dot{B}^{\rho+2}_{2,r})}
     +\|c^h\|_{\widetilde{L}^\infty_t(\dot{B}^{\rho+1}_{2,r})\cap
     \widetilde{L}^1_t(\dot{B}^{\rho+1}_{2,r})}
     +\|v^h\|_{\widetilde{L}^\infty_t(\dot{B}^{\rho}_{2,r})\cap
     \widetilde{L}^1_t(\dot{B}^{\rho+2}_{2,r})}\\
     &\leq& Ce^{C\widetilde{U}(t)}\big(
      \|(v_0,c_0)^l \|_{\dot{B}^\rho_{2,r}}
      +\|c_0^h \|_{\dot{B}^{\rho+1}_{2,r}}
            +\|v_0^h \|_{\dot{B}^{\rho}_{2,r}}
            +\|(L,G)^l
            \|_{\widetilde{L}^1_t(\dot{B}^{\rho}_{2,r})}\\
            &&
     +\|L^h \|_{\widetilde{L}^1_t(\dot{B}^{\rho+1}_{2,r})}
          +\|G^h\|_{\widetilde{L}^1_t(\dot{B}^{\rho}_{2,r})}
      \big),
    \end{eqnarray*}
where $C$ depends only on $N$ and $\rho$.
\end{prop}

Then, we will study the following mixed linear system
    \begin{equation}
  \left\{\begin{array}{l}
    \nabla \cdot v=\nabla \cdot c=0, \ \ x\in \mathbb{R}^N,\ N\geq2,\\
      v_{t} -\mu\Delta v -\Lambda c =0,\\
                     c_t +\Lambda v =0,\\
                (v,c)(0,x)=(v_0,c_0)(x).
  \end{array}
  \right. \label{vis2-E4.1-11}
\end{equation}

Using the similar arguments as that in \cite{Chen} (Lemma 4.1--4.2, Proposition 4.4), we have the following lemma.
\begin{lem}
  Let $\mathcal{G}$ be the Green matrix of the following system
       \begin{equation}
  \left\{\begin{array}{l}
              c_t +\Lambda v =0,\\
               v_{t} -\mu\Delta v -\Lambda c =0,
                  \end{array}
  \right.
\end{equation}
then, we have the following explicit expression for $\hat{\mathcal{G}}$:
    \begin{equation}
      \hat{\mathcal{G}}=\left[
      \begin{array}{ll}
        \frac{\lambda_+e^{\lambda_-t}-\lambda_-e^{\lambda_+t}}{\lambda_+-\lambda_-}I_{N\times N}
        &-\frac{ e^{\lambda_+t}- e^{\lambda_-t}}{\lambda_+-\lambda_-}|\xi|I_{N\times N}\\
       -\frac{ e^{\lambda_+t}- e^{\lambda_-t}}{\lambda_+-\lambda_-}|\xi|I_{N\times N}
        & \frac{\lambda_+e^{\lambda_+t}-\lambda_-e^{\lambda_-t}}{\lambda_+-\lambda_-}I_{N\times N}
      \end{array}
      \right],
    \end{equation}
    where
    $$\lambda_\pm=-\frac{1}{2}\mu|\xi|^2\pm \frac{1}{2}\sqrt{\mu^2|\xi|^4-4|\xi|^2}.$$
\end{lem}
 \begin{lem}
    There exist positive constant $R$ and $\vartheta$ depending on $\mu$ such that
        \begin{equation}
    |\partial^\alpha_\xi \hat{\mathcal{G}}(\xi,t)|\leq C_{|\alpha|} e^{-\vartheta|\xi|^2t}(1+|\xi|)^\alpha
    (1+t)^{|\alpha|},\ \forall\ |\xi|\leq R,
        \end{equation}
        \begin{eqnarray}
          \hat{\mathcal{G}}(\xi,t)&=&e^{-\mu^{-1}t}\left[
      \begin{array}{ll}
        I_{N\times N}
        &0\\
      0
        & 0
      \end{array}
      \right]+e^{-\mu|\xi|^2t}\left[
      \begin{array}{ll}
       0
        &0\\
      0
        &  I_{N\times N}
      \end{array}
      \right]\nonumber\\
      &&+\hat{\mathcal{G}}^1(\xi,t)\left[
      \begin{array}{ll}
       0
        & I_{N\times N}\\
    I_{N\times N}
        &    0
      \end{array}
      \right]+\hat{\mathcal{G}}^2(\xi,t)\left[
      \begin{array}{ll}
       I_{N\times N}
        & 0\\
    0
        &    I_{N\times N}
      \end{array}
      \right],\ \forall\ |\xi|\geq R,
        \end{eqnarray}
      where $\hat{\mathcal{G}}^1$ and $\hat{\mathcal{G}}^2$ satisfy the estimates
        \begin{equation}
          |\partial^\alpha_\xi \hat{\mathcal{G}}^1|\leq C|\xi|^{-|\alpha|-1}(e^{-\frac{t}{2\mu}}+e^{-\vartheta |\xi|^2t}),
        \end{equation}
          \begin{equation}
          |\partial^\alpha_\xi \hat{\mathcal{G}}^2|\leq C|\xi|^{-|\alpha|-2}(e^{-\frac{t}{2\mu}}+e^{-\vartheta |\xi|^2t}).
        \end{equation}
\end{lem}

\begin{prop}\label{vis2-P4.2}
  Let $\mathcal{C}$ be a ring centered at $0$ in $\mathbb{R}^N$. Then there exist positive constants $R_0$, $C$, $\vartheta$ depending on $\mu$ such that, if $\mathrm{supp}\hat{u}\subset \lambda \mathcal{C}$, then we have
    \begin{description}
      \item[(a)] if $\lambda \leq R_0$, then
        \begin{equation}
          \|\mathcal{G}*u\|_{L^2}\leq C e^{-\vartheta\lambda^2t}\|u\|_{L^2};
        \end{equation}
        \item[(b)] if $b\leq \lambda \leq R_0$, then for any $1\leq p\leq \infty$,
        \begin{equation}
          \|\mathcal{G}*u\|_{L^p}\leq C(1+b^{-N-1}) e^{-\vartheta\lambda^2t}\|u\|_{L^p};
        \end{equation}
        \item[(c)] if $\lambda >R_0$, then for any $1\leq p\leq \infty$,
        \begin{equation}
          \|\mathcal{G}^1*u\|_{L^p}\leq C \lambda^{-1}e^{-\vartheta t}\|u\|_{L^p},
        \end{equation}
                \begin{equation}
          \|\mathcal{G}^2*u\|_{L^p}\leq C \lambda^{-2}e^{-\vartheta t}\|u\|_{L^p}.
        \end{equation}
    \end{description}
\end{prop}

From Proposition \ref{vis2-P4.2}, or using the similar arguments as that in  \cite{Danchin2010} (lemma 2), we can obtain the following
lemma.
\begin{lem}\label{vis2-L4.1}
  Let $(v,c)$ be a solution of (\ref{vis2-E4.1-11}).
There exist two constants $\vartheta$ and $C$ such that, if
$2^j >R_0$, then for all $p\in[1,\infty]$,
    \begin{equation}
      \|\Delta_j v(t)\|_{L^p}\leq C\left(
      2^{-j}e^{-\vartheta t}\|\Delta_j c_0\|_{L^p}
      +\left(e^{-\vartheta t2^{2j}}+  2^{-2j}e^{-\vartheta t}
      \right)\|\Delta_j v_0\|_{L^p}
      \right),
    \end{equation}
 \begin{equation}
      \|\Delta_j c(t)\|_{L^p}\leq Ce^{-\vartheta t}\left(
      \|\Delta_j c_0\|_{L^p}
      + 2^{-j}\|\Delta_j v_0\|_{L^p}
      \right).
    \end{equation}
 For all $m\geq1$, there exist two constants $\alpha$ and $C$ such
 that, if $2^j \leq R_0$ then
     \begin{equation}
    e^{\vartheta t 2^{2j}}  \|(\Delta_j c,\Delta_j v)(t)\|_{L^2}
    \leq C
      \|(\Delta_j c_0,\Delta_j v_0)\|_{L^2}.
    \end{equation}
\end{lem}

\section{A priori estimates}\label{vis2-s5}
Assume that $v,E$  be a solution of (\ref{vis2-E1.4}). Letting
    \begin{equation}
      X_{p,0}=Y_{s,0}+Z_{p,0},
    \end{equation}
    \begin{equation}
      Y_{s,0}=\|E_0^l\|_{\dot{B}^s_{2,r}}
      +\|E_0^h\|_{\dot{B}^{s+1}_{2,r}}
      +\|v_0\|_{\dot{B}^s_{2,r}},
    \end{equation}
        \begin{equation}
      Z_{p,0}= \|E_0^h\|_{ \dot{B}^{\frac{N}{p}}_{p,1}}
      +\|v_0^h\|_{\dot{B}^{\frac{N}{p}-1}_{p,1}},
    \end{equation}
and
        \begin{equation}
      X_{p}(t)=Y_s(t)+Z_p(t),
    \end{equation}
        \begin{equation}
      Y_{s}(t)=\|E^l\|_{\widetilde{L}^2_t(\dot{B}^{s+1}_{2,r})\cap \widetilde{L}^1_t(\dot{B}^{s+2}_{2,r})}
      +\|E\|_{\widetilde{L}^\infty_t(\dot{B}^{s+1}_{2,r} )}
      +\|E^h\|_{\widetilde{L}^1_t(\dot{B}^{s+1}_{2,r} )}
            +\|v \|_{\widetilde{L}^\infty_t(\dot{B}^s_{2,r})\cap \widetilde{L}^1_t(\dot{B}^{s+2}_{2,r})},
    \end{equation}
        \begin{equation}
      Z_{p}(t)=
      \|E^h\|_{\widetilde{L}^\infty_t( \dot{B}^{\frac{N}{p}}_{p,1})}
      +\|E^h\|_{\widetilde{L}^1_t(  \dot{B}^{\frac{N}{p}}_{p,1})}
      +\|v^h\|_{\widetilde{L}^\infty_t(\dot{B}^{\frac{N}{p}-1}_{p,1})\cap \widetilde{L}^1_t(\dot{B}^{\frac{N}{p}+1}_{p,1})}.
    \end{equation}
    In this section, we will prove the following \textbf{Claim 1}:

There exist two positive constants $\lambda_1$ and $C_0$ such that,
if
    \begin{equation}
      X_{p_2}(t)\leq \lambda_1, \textrm{ for all }t\in[0,T],\label{vis2-E5.7}
    \end{equation}
then
        \begin{equation}
      X_{p_i}(t)\leq C_0  X_{p_i,0}, \textrm{ for all }i=1,2\ t\in[0,T].
    \end{equation}

    \begin{lem}\label{vis2-L5.1}
      Under the condition    (\ref{vis2-E5.7}), we have
         \begin{equation}
      Y_{s}(t)\leq C_1  Y_{s,0}, \textrm{ for all }\ t\in[0,T].
    \end{equation}
    \end{lem}
    \begin{proof}
   Let
      $$
    c=\Lambda^{-1}{ \nabla\cdot E},
    $$ where $\Lambda^{s} f=  \mathcal{F}^{-1} (|\xi|^s \hat{f})$.
Then, the system (\ref{vis2-E1.4}) reads
    \begin{equation}
  \left\{\begin{array}{l}
    \nabla \cdot v=\nabla \cdot c=0, \ \ x\in \mathbb{R}^N,\ N\geq2,\\
        v_{it}+T_v\cdot \nabla v_i+\partial_i p-\mu\Delta v_i-\Lambda c_i=G_i:=-\partial_jT'_{  v_i}  v_j+ E_{jk}\partial_j E_{ik},\\
            c_t+T_v\cdot \nabla c +\Lambda v=L:=-\partial_jT'_{  c} v_j+\Lambda^{-1} \nabla\cdot(\nabla vE)
            -[\Lambda^{-1}\nabla\cdot, v\cdot ]\nabla E,\\
                     \Delta E_{ij}= \Lambda\partial_j c_i+\partial_k(\partial_k E_{ij}-\partial_jE_{ik }),\\
                (v,c)(0,x)=(v_0,\Lambda^{-1}\nabla\cdot E_0)(x).
  \end{array}
  \right.\label{vis2-E5.5}
\end{equation}
From Proposition \ref{vis2-P4.3}, we have
     \begin{eqnarray}
     &&\|(v,c)^l \|_{\widetilde{L}^\infty_t(\dot{B}^s_{2,r})\cap
     \widetilde{L}^1_t(\dot{B}^{s+2}_{2,r})}
     +\|c^h\|_{\widetilde{L}^\infty_t(\dot{B}^{s+1}_{2,r})\cap
     \widetilde{L}^1_t(\dot{B}^{s+1}_{2,r})}
     +\|v^h\|_{\widetilde{L}^\infty_t(\dot{B}^{s}_{2,r})\cap
     \widetilde{L}^1_t(\dot{B}^{s+2}_{2,r})}\nonumber\\
     &\leq& Ce^{C\widetilde{U}(t)}\big(
      \|(v_0,c_0)^l \|_{\dot{B}^s_{2,r}}
      +\|c_0^h \|_{\dot{B}^{s+1}_{2,r}}
            +\|v_0^h \|_{\dot{B}^{s}_{2,r}}
            +\|L^l
            \|_{\widetilde{L}^1_t(\dot{B}^{s}_{2,r})}\nonumber\\
            &&
     +\|L^h \|_{\widetilde{L}^1_t(\dot{B}^{s+1}_{2,r})}
          +\|G\|_{\widetilde{L}^1_t(\dot{B}^{s}_{2,r})}
      \big)\nonumber\\
     &\leq& C \big(
      Y_{s,0}      +\|L^l
            \|_{\widetilde{L}^1_t(\dot{B}^{s}_{2,r})}
                 +\|L^h \|_{\widetilde{L}^1_t(\dot{B}^{s+1}_{2,r})}
          +\|G \|_{\widetilde{L}^1_t(\dot{B}^{s}_{2,r})}
      \big),\label{vis2-E5.10-1}
    \end{eqnarray}
where $\widetilde{U}(t)=\int^t_0\|\nabla v(\tau)\|_{L^{\infty}}d\tau\leq CX_{p_2}(t)\leq C$, when $\lambda_1\leq 1$.
From Proposition \ref{vis2-P3.1-0} and $s>-1$, we have, for all $t\in[0,T]$,
    \begin{eqnarray}
      \|\partial_jT'_{v_i}v_j\|_{\widetilde{L}^1_t(\dot{B}^{s }_{2,r})}
      &\leq& \|T _{\partial_j v_i}v_j\|_{\widetilde{L}^1_t(\dot{B}^{s }_{2,r})}
      +\|R( v_i,v_j)\|_{\widetilde{L}^1_t(\dot{B}^{s+1 }_{2,r})}\nonumber\\
      &      \leq& C\|  v\|_{L^1_t(\dot{B}^{\frac{N}{p_2}+1}_{p_2,1})}\|v\|_{\widetilde{L}^\infty_t(
      \dot{B}^{s}_{2,r})}
       \leq C\lambda_1\|v\|_{\widetilde{L}^\infty_t(
      \dot{B}^{s}_{2,r})}.\label{vis2-E5.11}
    \end{eqnarray}
From (\ref{vis2-E3.1}) and $\nabla\cdot E^\top=0$, we have  under the assumption that $s+1>0$, $s+1<\frac{N}{2}$ (if $r>1$) or $s+1\leq\frac{N}{2}$ (if $r=1$),
\begin{equation}
      \|E_{jk}\partial_j E_{ik}\|_{\widetilde{L}^1_t(\dot{B}^{s }_{2,r})}= \|\partial_j (E_{jk} E_{ik})\|_{\widetilde{L}^1_t(\dot{B}^{s }_{2,r})}
      \leq C\|E\|_{L^2_t(L^\infty)}\|E\|_{\widetilde{L}^2_t(
      \dot{B}^{s+1}_{2,r})}
       \leq C\lambda_1\|E\|_{\widetilde{L}^2_t(\dot{B}^{s+1}_{2,r})}.\label{vis2-E5.12}
    \end{equation}
From (\ref{vis2-E5.11})--(\ref{vis2-E5.12}), we have
            \begin{equation}
      \| G\|_{\widetilde{L}^1_t(\dot{B}^{s }_{2,r})}
       \leq C\lambda_1(\|v\|_{\widetilde{L}^\infty_t(
      \dot{B}^{s}_{2,r})}
      +\|E\|_{\widetilde{L}^2_t(\dot{B}^{s+1}_{2,r})}).\label{vis2-E5.13}
    \end{equation}
From Proposition \ref{vis2-P3.1-0}, under the assumption that $s+1>0$, $s+1<\frac{N}{2}$ (if $r>1$) or $s+1\leq\frac{N}{2}$ (if $r=1$), we have, for all $t\in[0,T]$,
    \begin{equation}
      \|(\partial_jT'_cv_j)^l\|_{\widetilde{L}^1_t(\dot{B}^s_{2,r})}
      \leq C\|( T'_cv)^l\|_{\widetilde{L}^1_t(\dot{B}^{s+1}_{2,r})}
      \leq C\|c\|_{L^2_t(L^\infty)}\|v\|_{\widetilde{L}^2_t(\dot{B}^{s+1}_{2,r})}
      \leq C\lambda_1\|v\|_{\widetilde{L}^2_t(\dot{B}^{s+1}_{2,r})},\label{vis2-E5.14}
    \end{equation}
        \begin{equation}
      \|(\partial_jT'_cv_j)^h\|_{\widetilde{L}^1_t(\dot{B}^{s+1}_{2,r})}
      \leq C\|c\|_{L^\infty_t(L^\infty)}\|v\|_{\widetilde{L}^1_t(\dot{B}^{s+2}_{2,r})}
       \leq C\lambda_1\|v\|_{\widetilde{L}^1_t(\dot{B}^{s+2}_{2,r})}.\label{vis2-E5.15}
    \end{equation}
where we use the following estimates
    $$
    \|c\|_{L^\infty}\leq \|c^l\|_{L^\infty}+\|c^h\|_{L^\infty}
    \leq C\left(
    \|c^l\|_{\dot{B}^\frac{N}{2}_{2,1}}+\|c^h\|_{\dot{B}^\frac{N}{p_2}_{p_2,1}}
    \right)\leq C\left(
    \|E^l\|_{\dot{B}^\frac{N}{2}_{2,1}}+\|E^h\|_{\dot{B}^\frac{N}{p_2}_{p_2,1}}
    \right)
        $$
        and
            $$
            \|c\|_{L^2_t(L^\infty)}+\|c\|_{L^\infty_t(L^\infty)}\leq CX_{p_2}(t)\leq C\lambda_1.
            $$
Similarly, from Proposition \ref{vis2-P3.1-0}, we have,
    \begin{equation}
    \|( T_E\nabla v)^l\|_{\widetilde{L}^1_t(\dot{B}^s_{2,r})}
    \leq C\|E\|_{L^2_t(L^\infty)}
    \|\nabla v\|_{\widetilde{L}^2_t(\dot{B}^s_{2,r})}
     \leq C \lambda_1
    \|\nabla v\|_{\widetilde{L}^2_t(\dot{B}^s_{2,r})},\label{vis2-E5.16}
    \end{equation}
    \begin{equation}
    \|( T_E\nabla v)^h\|_{\widetilde{L}^1_t(\dot{B}^{s+1}_{2,r})}
    \leq C\|E\|_{L^\infty_t(L^\infty)}
    \|\nabla v\|_{\widetilde{L}^1_t(\dot{B}^{s+1}_{2,r})}\leq C \lambda_1
    \|\nabla v\|_{\widetilde{L}^1_t(\dot{B}^{s+1}_{2,r})},
    \end{equation}
    \begin{eqnarray}
    \|( T_{\nabla v}E)^l\|_{\widetilde{L}^1_t(\dot{B}^s_{2,r})}
    &\leq& C\|\nabla v\|_{L^2_t(\dot{B}^{-1}_{\infty,\infty})}
    \|E\|_{\widetilde{L}^2_t(\dot{B}^{s+1}_{2,r})}\nonumber\\
    &\leq& C\lambda_1
    \|E\|_{\widetilde{L}^2_t(\dot{B}^{s+1}_{2,r})},
    \end{eqnarray}
     \begin{eqnarray}
    \|( T_{\nabla v}E)^h\|_{\widetilde{L}^1_t(\dot{B}^{s+1}_{2,r})}
    &\leq& C\|\nabla v\|_{L^1_t(L^{\infty})}
    \|E\|_{\widetilde{L}^\infty_t(\dot{B}^{s+1}_{2,r})}\nonumber\\
    &\leq& C\lambda_1
   ( \|E^l\|_{\widetilde{L}^\infty_t(\dot{B}^{s}_{2,r})}+\|E^h\|_{\widetilde{L}^\infty_t(\dot{B}^{s+1}_{2,r})}),
    \end{eqnarray}
      \begin{eqnarray}
    \|( R( E_{jk},\partial_j v_i))^l\|_{\widetilde{L}^1_t(\dot{B}^s_{2,r})}
    &=&\|\partial_j( R( E_{jk}, v_i))^l\|_{\widetilde{L}^1_t(\dot{B}^s_{2,r})}\nonumber\\
    &\leq&  \|( R( E, v))^l\|_{\widetilde{L}^1_t(\dot{B}^{s+\frac{N}{p_2}+1}_{\frac{2p_2}{2+p_2},r})}\nonumber\\
    &\leq&
    C\|  v\|_{L^2_t(\dot{B}^{\frac{N}{p_2}}_{p_2, \infty})}
    \|E\|_{\widetilde{L}^2_t(\dot{B}^{s+1}_{2,r})}\nonumber\\
    &\leq& C\lambda_1
    \|E\|_{\widetilde{L}^2_t(\dot{B}^{s+1}_{2,r})},
    \end{eqnarray}
     \begin{eqnarray}
    \|( R( E,\nabla v))^h\|_{\widetilde{L}^1_t(\dot{B}^{s+1}_{2,r})}
    &\leq&  \|( R( E,\nabla v))^l\|_{\widetilde{L}^1_t(\dot{B}^{s+\frac{N}{p_2}+1}_{\frac{2p_2}{2+p_2},r})}\nonumber\\
    &\leq&
    C\|\nabla v\|_{L^1_t(\dot{B}^{\frac{N}{p_2}}_{p_2, \infty})}
    \|E\|_{\widetilde{L}^\infty_t(\dot{B}^{s+1}_{2,r})}\nonumber\\
    &\leq& C\lambda_1
   ( \|E^l\|_{\widetilde{L}^\infty_t(\dot{B}^{s}_{2,r})}
   +\|E^h\|_{\widetilde{L}^\infty_t(\dot{B}^{s+1}_{2,r})}),\label{vis2-E5.21}
    \end{eqnarray}
    where we use the following estimates
        $$
        \|\nabla v\|_{L^2_t(\dot{B}^{-1}_{\infty,\infty})}
        \leq C\|\nabla v\|_{L^2_t(\dot{B}^{\frac{N}{p_2}-1}_{p_2,1})}\leq CX_{p_2}(t)\leq C\lambda_1,
        $$
        $$
        \|\nabla v\|_{L^1_t(L^{\infty})}
        \leq C\|\nabla v\|_{L^1_t(\dot{B}^{\frac{N}{p_2}}_{p_2,1})}\leq CX_{p_2}(t)\leq C\lambda_1.
        $$
From (\ref{vis2-E5.16})--(\ref{vis2-E5.21}), we obtain
             \begin{equation}
    \|( \Lambda^{-1} \nabla\cdot(\nabla vE))^l\|_{\widetilde{L}^1_t(\dot{B}^s_{2,r})}
    \leq  C \lambda_1
   ( \|\nabla v\|_{\widetilde{L}^2_t(\dot{B}^s_{2,r})}+ \|E\|_{\widetilde{L}^2_t(\dot{B}^{s+1}_{2,r})}),\label{vis2-E5.22}
    \end{equation}
     \begin{equation}
    \|( \Lambda^{-1} \nabla\cdot(\nabla vE))^h\|_{\widetilde{L}^1_t(\dot{B}^{s+1}_{2,r})}
    \leq  C \lambda_1
   ( \|\nabla v\|_{\widetilde{L}^1_t(\dot{B}^{s+1}_{2,r})}+ \|E^l\|_{\widetilde{L}^\infty_t(\dot{B}^{s}_{2,r})}+\|E^h\|_{\widetilde{L}^\infty_t(\dot{B}^{s+1}_{2,r})}).
    \end{equation}
    Similar, using the classical commutator estimate (Lemma 2.97 in \cite{Bahouri}), we obtain
      \begin{equation}
    \|( [\Lambda^{-1}\nabla\cdot, v\cdot ]\nabla E)^l\|_{\widetilde{L}^1_t(\dot{B}^s_{2,r})}
    \leq  C \lambda_1
   ( \|\nabla v\|_{\widetilde{L}^2_t(\dot{B}^s_{2,r})}+ \|E\|_{\widetilde{L}^2_t(\dot{B}^{s+1}_{2,r})}),
    \end{equation}
     \begin{equation}
    \|( [\Lambda^{-1}\nabla\cdot, v\cdot ]\nabla E)^h\|_{\widetilde{L}^1_t(\dot{B}^{s+1}_{2,r})}
    \leq  C \lambda_1
   ( \|\nabla v\|_{\widetilde{L}^1_t(\dot{B}^{s+1}_{2,r})}+ \|E^l\|_{\widetilde{L}^\infty_t(\dot{B}^{s}_{2,r})}
   +\|E^h\|_{\widetilde{L}^\infty_t(\dot{B}^{s+1}_{2,r})}).\label{vis2-E5.25}
    \end{equation}
From (\ref{vis2-E5.14})--(\ref{vis2-E5.15}) and (\ref{vis2-E5.22})--(\ref{vis2-E5.25}),  we get
    \begin{equation}
    \|L^l\|_{\widetilde{L}^1_t(\dot{B}^s_{2,r})}
    \leq  C \lambda_1
   ( \|  v\|_{\widetilde{L}^2_t(\dot{B}^{s+1}_{2,r})}+ \|E\|_{\widetilde{L}^2_t(\dot{B}^{s+1}_{2,r})}),\label{vis2-E5.26}
    \end{equation}
     \begin{equation}
    \|L^h\|_{\widetilde{L}^1_t(\dot{B}^{s+1}_{2,r})}
    \leq  C \lambda_1
   ( \|  v\|_{\widetilde{L}^1_t(\dot{B}^{s+2}_{2,r})}+ \|E^l\|_{\widetilde{L}^\infty_t(\dot{B}^{s}_{2,r})}
   +\|E^h\|_{\widetilde{L}^\infty_t(\dot{B}^{s+1}_{2,r})}).\label{vis2-E5.27}
    \end{equation}

From (\ref{vis2-E5.10-1}), (\ref{vis2-E5.13}) and (\ref{vis2-E5.26})--(\ref{vis2-E5.27}), we have
    \begin{eqnarray}
     &&\|(v,c)^l \|_{\widetilde{L}^\infty_t(\dot{B}^s_{2,r})\cap
     \widetilde{L}^1_t(\dot{B}^{s+2}_{2,r})}
     +\|c^h\|_{\widetilde{L}^\infty_t(\dot{B}^{s+1}_{2,r})\cap
     \widetilde{L}^1_t(\dot{B}^{s+1}_{2,r})}
     +\|v^h\|_{\widetilde{L}^\infty_t(\dot{B}^{s}_{2,r})\cap
     \widetilde{L}^1_t(\dot{B}^{s+2}_{2,r})}\nonumber\\
          &\leq& C
      Y_{s,0}      +C\lambda_1 Y_s(t).\label{vis2-E5.28}
    \end{eqnarray}
    From Lemmas \ref{vis2-L2.2}--\ref{vis2-L2.3}, we have
    \begin{equation}
    \partial_m E_{ij}-\partial_jE_{im}
       =E_{lj}\partial_lE_{im}-E_{lm}\partial_lE_{ij}
       =\partial_l(E_{lj}E_{im}-E_{lm}E_{ij}).\label{vis2-E7.1}
  \end{equation}
  and from (\ref{vis2-E5.5})$_4$, Proposition \ref{vis2-P3.1}, Lemma \ref{vis2-L3.2},
  \begin{eqnarray*}
   &&\|E \|_{\widetilde{L}^\infty_t(\dot{B}^{s+1}_{2,r})}
   +\|E^l \|_{\widetilde{L}^2_t(\dot{B}^{s+1}_{2,r})\cap
     \widetilde{L}^1_t(\dot{B}^{s+2}_{2,r})}
     +\|E^h \|_{
     \widetilde{L}^1_t(\dot{B}^{s+1}_{2,r})}\\
     &\leq&
   C \|c \|_{\widetilde{L}^\infty_t(\dot{B}^{s+1}_{2,r})}
   +C\|c^l \|_{\widetilde{L}^2_t(\dot{B}^{s+1}_{2,r})\cap
     \widetilde{L}^1_t(\dot{B}^{s+2}_{2,r})}
     +C\|c^h \|_{
     \widetilde{L}^1_t(\dot{B}^{s+1}_{2,r})}\\&&
     +C \|E^2 \|_{\widetilde{L}^\infty_t(\dot{B}^{s+1}_{2,r})}
   +C\|(E^2)^l \|_{\widetilde{L}^2_t(\dot{B}^{s+1}_{2,r})\cap
     \widetilde{L}^1_t(\dot{B}^{s+2}_{2,r})}
     +C\|(E^2)^h \|_{
     \widetilde{L}^1_t(\dot{B}^{s+1}_{2,r})}\\
      &\leq& C \|c \|_{\widetilde{L}^\infty_t(\dot{B}^{s+1}_{2,r})}
   +C\|c^l \|_{\widetilde{L}^2_t(\dot{B}^{s+1}_{2,r})\cap
     \widetilde{L}^1_t(\dot{B}^{s+2}_{2,r})}
     +C\|c^h \|_{
     \widetilde{L}^1_t(\dot{B}^{s+1}_{2,r})}\\
   &&  +C \|E  \|_{L^\infty_t(\dot{B}^\frac{N}{p_2}_{p_2,1})}
   (  \|E \|_{\widetilde{L}^\infty_t(\dot{B}^{s+1}_{2,r})}
   + \|E^l \|_{\widetilde{L}^2_t(\dot{B}^{s+1}_{2,r})\cap
     \widetilde{L}^1_t(\dot{B}^{s+2}_{2,r})}
     + \|E^h \|_{
     \widetilde{L}^1_t(\dot{B}^{s+1}_{2,r})})   \\
      &\leq&  C \|c \|_{\widetilde{L}^\infty_t(\dot{B}^{s+1}_{2,r})}
   +C\|c^l \|_{\widetilde{L}^2_t(\dot{B}^{s+1}_{2,r})\cap
     \widetilde{L}^1_t(\dot{B}^{s+2}_{2,r})}
     +C\|c^h \|_{
     \widetilde{L}^1_t(\dot{B}^{s+1}_{2,r})}\\
   &&  + C\lambda_1
      (  \|E \|_{\widetilde{L}^\infty_t(\dot{B}^{s+1}_{2,r})}
   + \|E^l \|_{\widetilde{L}^2_t(\dot{B}^{s+1}_{2,r})\cap
     \widetilde{L}^1_t(\dot{B}^{s+2}_{2,r})}
     + \|E^h \|_{
     \widetilde{L}^1_t(\dot{B}^{s+1}_{2,r})}) .
  \end{eqnarray*}
When $\lambda_1$ is small enough, $C\lambda_1<\frac{1}{2}$, we have
 \begin{eqnarray}
   &&\|E \|_{\widetilde{L}^\infty_t(\dot{B}^{s+1}_{2,r})}
   +\|E^l \|_{\widetilde{L}^2_t(\dot{B}^{s+1}_{2,r})\cap
     \widetilde{L}^1_t(\dot{B}^{s+2}_{2,r})}
     +\|E^h \|_{
     \widetilde{L}^1_t(\dot{B}^{s+1}_{2,r})}\nonumber\\
      &\leq&  C \|c \|_{\widetilde{L}^\infty_t(\dot{B}^{s+1}_{2,r})}
   +C\|c^l \|_{\widetilde{L}^2_t(\dot{B}^{s+1}_{2,r})\cap
     \widetilde{L}^1_t(\dot{B}^{s+2}_{2,r})}
     +C\|c^h \|_{
     \widetilde{L}^1_t(\dot{B}^{s+1}_{2,r})}.\label{vis2-E5.30}
  \end{eqnarray}
  From (\ref{vis2-E5.28}) and (\ref{vis2-E5.30}), we get
  \begin{equation*}
      Y_s(t)\leq C
      Y_{s,0}      +C\lambda_1 Y_s(t).
    \end{equation*}
   When $\lambda_1$ is small enough, $C\lambda_1<\frac{1}{2}$, we obtain
     $
      Y_s(t)\leq C
      Y_{s,0}     .
    $
  \end{proof}

    \begin{lem}\label{vis2-L5.2}
      Under the condition    (\ref{vis2-E5.7}), we have
         \begin{equation}
      Z_{p }(t)\leq C_2 ( Z_{p ,0}+ Y_{s,0}), \textrm{ for all }p=p_1,p_2\ t\in[0,T].   \label{vis2-E5.48-0}
    \end{equation}
    \end{lem}
    \begin{proof}
    Applying operator $\Delta_q$ to  (\ref{vis2-E1.4}), we obtain the following system for $(v_q,E_q):=(\Delta_q v,\Delta_q E)$,
\begin{equation}
  \left\{\begin{array}{l}
    \nabla \cdot v_q=0, \ \ x\in \mathbb{R}^N,\ N\geq 2,\\
        \partial_tv_{q,i}+S_{q-1}v\cdot \nabla v_{q,i}-\mu\Delta v_{q,i}-\partial_jE_{q,ij}=G_{q,i}\\:=S_{q-1}v\cdot\nabla v_{q,i}-\Delta_q(T_v\cdot\nabla v_i)+\Delta_q(-T'_{\nabla v_i}\cdot v-\partial_i p+ E_{jk}\partial_j E_{ik}),\\
           \partial_t E_{q}+S_{q-1}v\cdot \nabla E_q-\nabla v_q=L_q\\
            :=\partial_i(S_{q-1}v_i E_q-\Delta_q T_{v_i}E) -\Delta_q\partial_i T'_{E}v_i+\Delta_q(\nabla vE),\\
                (v_q,E_q)(0,x)=(\Delta_qv_0,\Delta_qE_0)(x).
  \end{array}
  \right.
\end{equation}

Since
$$L_q=\Delta_q L+\partial_i\left(
\sum_{q'\sim q}(S_{q-1}v_i-S_{q'-1}v_i)\Delta_q\Delta_{q'} E+[S_{q'-1}v_i,\Delta_q]\Delta_{q'}E
\right),$$
where $L:=\nabla vE-\partial_iT'_{E}v_i$, using Lemma \ref{vis2-L3.1} and  the classical commutator estimate (Lemma 2.97 in \cite{Bahouri}), we have
\begin{equation}
  \|\nabla L_q\|_{L^p}\leq \|\nabla\Delta_q L\|_{L^p}+C\|\nabla v\|_{L^\infty}
  \sum_{q'\sim q}\|\nabla \Delta_{q'}E\|_{L^p}.\label{vis2-E5.34-1}
\end{equation}
Similar, we get
    \begin{equation}
  \|G_q\|_{L^p}\leq \| \Delta_q G\|_{L^p}+C\|\nabla v\|_{L^\infty}
  \sum_{q'\sim q}\|  \Delta_{q'}v\|_{L^p},\label{vis2-E5.35-1}
\end{equation}
where $G:=-T'_{\nabla v_i}\cdot v-\partial_i p+ E_{jk}\partial_j E_{ik}$.

In order to handle the convection terms, one may perform the Lagrangian change of variable $(\tau,x)=(t,\psi_q(t,y))$, where $\psi_q$ stands for the flow of $S_{q-1}v$. Let $\phi_q:=\psi_q^{-1}$, $\widetilde{f}_q=f_q\circ\psi_q$. Then, $(\widetilde{v}_q,\widetilde{E}_q):=(v_q\circ\psi_q,E_q\circ\psi_q)$ satisfies
     \begin{equation}
  \left\{\begin{array}{l}
    \nabla \cdot v_q=0, \ \ x\in \mathbb{R}^N,\ N\geq 2,\\
       \partial_t \widetilde{v}_{q }-\mu\Delta \widetilde{v}_{q }-\nabla\cdot\widetilde{E}_{q}
        =\widetilde{G}_q+R^1_q+R^2_q,\\
           \partial_t \widetilde{E}_{q}-\nabla \widetilde{v}_q=\widetilde{L}_q+R^3_q,\\
                (\widetilde{v}_q,\widetilde{E}_q)(0,x)
                =(\Delta_qv_0,\Delta_qE_0)(x),
  \end{array}
  \right.\label{vis2-E5.34}
\end{equation}
where
    $$
        R^1_{q,i}(t,x):=\partial_k\widetilde{E}_{q,ij}(t,x)(\partial_j \phi_{q,k}(t,\psi_q(t,x))-\delta_{jk}),
    $$
        \begin{eqnarray*}
          R^2_{q,i}&:=&\mu
                 (\partial_j \phi_{q,k}(t,\psi_q(t,x))-\delta_{jk})\partial_k\partial_l\widetilde{v}_{q,i}(t,x)\cdot \partial_j\phi_{q,l}(t,\psi_q(t,x))\\
          &&       + \mu \partial_j\partial_l\widetilde{v}_{q,i}(t,x) (\partial_j\phi_{q,l}(t,\psi_q(t,x))-\delta_{lj})         \\
            &&+ \mu\partial_j  \widetilde{v}_{q,i}(t,x) \Delta\phi_{q,j}(t,\psi_q(t,x)).
        \end{eqnarray*}
            $$
            R^3_{q,jk}:=\partial_l  \widetilde{v}_{qj}(t,x)\cdot (\partial_k\phi_{q,l}(t,\psi_q(t,x))-\delta_{lk}).
            $$
 Let
      $$
    \widetilde{c}_q=\Lambda^{-1}{ \nabla\cdot \widetilde{E}_q},
    $$ where $\Lambda^{s} f=  \mathcal{F}^{-1} (|\xi|^s \hat{f})$.
Then, the system (\ref{vis2-E5.34}) reads
    \begin{equation}
  \left\{\begin{array}{l}
         \partial_t \widetilde{v}_{q,i}-\mu\Delta \widetilde{v}_{q,i}-\Lambda\widetilde{c}_{q,i}
        =\widetilde{G}_q+R^1_q+R^2_q,\\
           \partial_t \widetilde{c}_{q}+\Lambda \widetilde{v}_q=\Lambda^{-1} \nabla\cdot\widetilde{L}_q+\Lambda^{-1} \nabla\cdot R^3_q,\\
                (\widetilde{v}_q,\widetilde{c}_q)(0,x)
                =(\Delta_qv_0,\Lambda^{-1} \nabla\cdot\Delta_qE_0)(x).
  \end{array}
  \right.\label{vis2-E5.35}
\end{equation}
Then we may write
    $$
        \left(\begin{array}{l}
         \Delta_j\widetilde{c}_q(t) \\
                    \Delta_j\widetilde{v}_q(t)
        \end{array}
        \right)=\mathcal{G}(\cdot,t)*\left(\begin{array}{l}
               \Delta_j\Lambda^{-1} \nabla\cdot\Delta_qE_0  \\
             \Delta_j\Delta_qv_0
        \end{array}
        \right)+\int^t_0 \mathcal{G}(\cdot,t-\tau)* \left(\begin{array}{l}
          \Delta_j\Lambda^{-1} \nabla\cdot(\widetilde{L}_q+ R^3_q)    \\
             \Delta_j(\widetilde{G}_q+R^1_q+R^2_q)
        \end{array}
        \right)d\tau
    $$
From Lemma \ref{vis2-L4.1}, we have, for all $2^j >R_0$, then for all $p\in[1,\infty]$,
    \begin{eqnarray}
     && \|\Delta_j\widetilde{v}_q(t)\|_{L^p}\nonumber\\
     &\leq&C  2^{- 2j} e^{-\vartheta t}\|\Delta_j\Delta_q \nabla E_0\|_{L^p}\nonumber\\
        &&
      +C\left(e^{-\vartheta 2^{2j}t}+  2^{-2j} e^{-\vartheta  t}
      \right)\|\Delta_j\Delta_qu_0\|_{L^p}\nonumber\\
        &&+C\int^t_0\big(
        2^{-2j} e^{-\vartheta (t-\tau)}(\|\Delta_j\nabla \widetilde{L}_q\|_{L^p}
        +\|\Delta_j\nabla R^3_q\|_{L^p})\label{vis2-E5.36}\\
        &&+ (e^{-\vartheta 2^{2j}(t-\tau)}+  2^{-2j}e^{-\vartheta (t-\tau)})
        (\|\Delta_j\widetilde{G}_q\|_{L^p}+\|\Delta_j  R^1_q\|_{L^p}+\|\Delta_j  R^2_q\|_{L^p})
        \big)d\tau,\nonumber
    \end{eqnarray}
\begin{eqnarray}
    &&\|\Delta_j\nabla \widetilde{c}_q\|_{L^p}\nonumber\\
        &\leq& Ce^{-\vartheta t}\left(
        \|\Delta_j\Delta_q \nabla E_0\|_{L^p}
        +\|\Delta_j\Delta_q v_0\|_{L^p}
        \right)\nonumber\\
    &&+C\int^t_0e^{-\vartheta (t-\tau)}
    \left(\|\Delta_j\nabla \widetilde{L}_q\|_{L^p}+\|\Delta_j\nabla R^3_q\|_{L^p}
    +\|\Delta_j  \widetilde{G}_q\|_{L^p}\right.\nonumber\\
    &&\left.+\|\Delta_j  R^1_q\|_{L^p}
    +\|\Delta_j R^2_q\|_{L^p}
    \right)d\tau.\label{vis2-E5.37}
\end{eqnarray}
We now have to fin suitable bounds for the right-hand side of (\ref{vis2-E5.36})--(\ref{vis2-E5.37}). In what follows we shall freely use the following estimate (Proposition 8 in \cite{Danchin2010}):
    \begin{equation}
      \|g\circ \psi_q\|_{L^p}=      \|  g\|_{L^p},
      \textrm{ for all function } g\in L^p,
    \end{equation}
        \begin{equation}
          \|\nabla \psi_q^{\pm1}\|_{L^\infty}\leq e^{C\widetilde{U}}\leq C,
        \end{equation}
         \begin{equation}
          \|\nabla \psi_q^{\pm1}-Id\|_{L^\infty}\leq e^{C\widetilde{U}}-1\leq e^{C\lambda_1}-1\leq C\lambda_1,
        \end{equation}
         \begin{equation}
          \|\nabla^k \psi_q^{\pm1}\|_{L^\infty}\leq C2^{(k-1)q}(e^{C\widetilde{U}}-1)\leq C2^{(k-1)q}\lambda_1,\textrm{ for }k=2,3,
        \end{equation}
   where $\widetilde{U}(t)=\int^t_0\|\nabla v(\tau)\|_{L^{\infty}}d\tau\leq CX_{p_2}(t)\leq C\lambda_1$.

  As for $\Delta_j\nabla \widetilde{L}_q$, we can use the fact that, owing to the chain rule,
    $$
    \Delta_j \nabla \widetilde{L}_q=\Delta_j(\partial_l L_q\circ\psi_q \nabla \psi_{q,l}).
    $$
 Then
    $$
    \|\Delta_j \nabla \widetilde{L}_q\|_{L^p}
    \leq C\|\nabla L_q\circ\psi_q\|_{L^p}\|\nabla \psi_q\|_{L^\infty}
     \leq C\|\nabla L_q \|_{L^p},
    $$
    combining the estimate (\ref{vis2-E5.34-1}), we get
         $$
    \|\Delta_j \nabla \widetilde{L}_q\|_{L^p}
    \leq C\|\nabla\Delta_q L\|_{L^p}+C\|\nabla v\|_{L^\infty}
  \sum_{q'\sim q}\|\nabla \Delta_{q'}E\|_{L^p}.
    $$
Using the Vishik's trick used in \cite{Vishik}, we can get
   \begin{equation}
    \|\Delta_j \nabla \widetilde{L}_q\|_{L^p}
    \leq C2^{q-j}\left(\|\nabla\Delta_q L\|_{L^p}+ \|\nabla v\|_{L^\infty}
  \sum_{q'\sim q}\|\nabla \Delta_{q'}E\|_{L^p}\right),\label{vis2-E5.44}
    \end{equation}
where we use the facts that
    $$
    \|\Delta_j\nabla \widetilde{L}_q\|_{L^p}\leq C2^{-j}\|\nabla^2(L_q\circ \psi_q)\|_{L^p},
    $$
$$
\partial_k\partial_l (L_q\circ\psi_q)=(\partial_m\partial_n L_q\circ\psi_q)\partial_k \psi_{q,m}
\partial_l \psi_{q,n}+(\partial_m  L_q\circ\psi_q)\partial_k\partial_l \psi_{q,m},
$$
  and
    $$
    \|\nabla^2(L_q\circ \psi_q)\|_{L^p}\leq C2^q\|\nabla L_q\|_{L^p}.
    $$
    Similar, we get
       \begin{equation}
    \|\Delta_j  \widetilde{G}_q\|_{L^p}
    \leq C2^{q-j}\left(\| \Delta_q G\|_{L^p}+ \|\nabla v\|_{L^\infty}
  \sum_{q'\sim q}\|  \Delta_{q'}v\|_{L^p}\right).\label{vis2-E5.45}
    \end{equation}

    Let us now turn to the study of $\nabla R_q^3$. As a preliminary step, we have
        $$
        \|\Delta_j\nabla R^3_q\|_{L^p}\leq C2^{-j}\|\Delta_j\nabla^2 R_q^3\|_{L^p}
        \leq C2^{-j}\| \nabla^2 R_q^3\|_{L^p}.
        $$
        Using the chain rule and H\"{o}lder inequality, we get
            \begin{eqnarray*}
              &&\|\nabla^2 R^3_q\|_{L^p}\\
                &\leq&C \|\nabla^3 \widetilde{v}_q\|\|\nabla\phi_q\circ\psi_q-Id\|_{L^\infty}
                +C \|\nabla^2 \widetilde{v}_q\|\|\nabla^2\phi_q\circ\psi_q \|_{L^\infty}\|\nabla \psi_q \|_{L^\infty}\\
                &&+C \|\nabla  \widetilde{v}_q\|\left(\|\nabla^3\phi_q\circ\psi_q \|_{L^\infty}\|\nabla \psi_q \|_{L^\infty}^2+\|\nabla^2\phi_q\circ\psi_q \|_{L^\infty}\|\nabla^2 \psi_q \|_{L^\infty}
                \right)
            \end{eqnarray*}
    and
        \begin{equation}
          \|\Delta_j\nabla R^3_q\|_{L^p}\leq C2^{3q-j}\lambda_1\|v_q\|_{L^p}.
        \end{equation}
 Similar, we have
       \begin{equation}
          \|\Delta_j R^1_q\|_{L^p}\leq C2^{ q-j}\lambda_1\|\nabla E_q\|_{L^p},
        \end{equation}
         \begin{equation}
          \|\Delta_j R^2_q\|_{L^p}\leq C2^{ 3q-j}\lambda_1\|v_q\|_{L^p}.\label{vis2-E5.48}
        \end{equation}

    From (\ref{vis2-E5.36})--(\ref{vis2-E5.37}), (\ref{vis2-E5.44})--(\ref{vis2-E5.48}), we have for $2^j>R_0$,
        \begin{eqnarray*}
          &&2^j\|\Delta_j  \widetilde{c}_q\|_{L^\infty_t(L^p)\cap L^1_t(L^p)}
          +\|\Delta_j \widetilde{v}_q\|_{L^\infty_t(L^p)}
          +2^{2j}\|\Delta_j \widetilde{v}_q\|_{ L^1_t(L^p)}\\
            &\lesssim &
            \|(\Delta_j\Delta_q \nabla E_0,\Delta_j\Delta_q v_0)\|_{L^p}
            +2^{q-j} \|(\Delta_q\nabla L,\Delta_q G)\|_{L^1_t(L^p)}\\
            &&+2^{q-j}\sum_{q'\sim q}\int^t_0\|\nabla v\|_{L^\infty}\|(\nabla E_{q'},v_{q'})\|_{L^p}d\tau
            +2^{q-j}\lambda_1\left(
            \|\nabla E_q\|_{L^1_t(L^p)}+2^{2q}\|v_q\|_{L^1_t(L^p)}
            \right).
        \end{eqnarray*}
Since
    \begin{equation}
      v_q=S_{q-N_0}\widetilde{v}_q\circ \phi_q+\sum_{j\geq q-N_0}\Delta_j \widetilde{v}_q\circ\phi_q,
    \end{equation}
        \begin{equation}
       c_q=S_{q-N_0}  \widetilde{c}_q\circ \phi_q+\sum_{j\geq q-N_0}\Delta_j  \widetilde{c}_q\circ\phi_q,
    \end{equation}
from the following estimate (\cite{Chen}, Lemma 2.3),
    \begin{equation}
    \|S_j(\Delta_q f\circ \psi_q^{\pm1})\|_{L^p}\leq C(\lambda_1+2^{j-q})\|\Delta_q f\|_{L^p},
    \end{equation}
    we have for  all $2^q> R_02^{ N_0}$,
         \begin{eqnarray*}
          &&2^q\|  {c}_q\|_{L^\infty_t(L^p)\cap L^1_t(L^p)}
          +\| {v}_q\|_{L^\infty_t(L^p)}
          +2^{2q}\| {v}_q\|_{ L^1_t(L^p)}\\
            &\leq &C2^{3N_0}\left(
            \|( \Delta_q \nabla E_0, \Delta_q v_0)\|_{L^p}
            +  \|(\Delta_q\nabla L,\Delta_q G)\|_{L^1_t(L^p)}
              + \sum_{q'\sim q}\int^t_0\|\nabla v\|_{L^\infty}\|(\nabla E_{q'},v_{q'})\|_{L^p}d\tau\right)\\
              &&
            +C 2^{3N_0}\lambda_1 \left(
            2^q\|  {E}_q\|_{L^\infty_t(L^p)\cap L^1_t(L^p)}
          +\| {v}_q\|_{L^\infty_t(L^p)}
          +2^{2q}\| {v}_q\|_{ L^1_t(L^p)}
            \right)\\
              &&
            +C 2^{-N_0}
            \left(
            2^q\|  {c}_q\|_{L^\infty_t(L^p)\cap L^1_t(L^p)}
          +\| {v}_q\|_{L^\infty_t(L^p)}
          +2^{2q}\| {v}_q\|_{ L^1_t(L^p)}
            \right).
        \end{eqnarray*}
            Choosing $N_0$ to be an integer such that $C2^{-N_0}\in(\frac{1}{4},\frac{1}{2})$, we have
            for  all $2^q> R_02^{ N_0}$,
         \begin{eqnarray*}
          &&2^q\|  {c}_q\|_{L^\infty_t(L^p)\cap L^1_t(L^p)}
          +\| {v}_q\|_{L^\infty_t(L^p)}
          +2^{2q}\| {v}_q\|_{ L^1_t(L^p)}\\
            &\leq &C \left(
            \|( \Delta_q \nabla E_0, \Delta_q v_0)\|_{L^p}
            +  \|(\Delta_q\nabla L,\Delta_q G)\|_{L^1_t(L^p)}
              + \sum_{q'\sim q}\int^t_0\|\nabla v\|_{L^\infty}\|(\nabla E_{q'},v_{q'})\|_{L^p}d\tau\right)\\
              &&
            +C  \lambda_1 \left(
            2^q\|  {E}_q\|_{L^\infty_t(L^p)\cap L^1_t(L^p)}
          +\| {v}_q\|_{L^\infty_t(L^p)}
          +2^{2q}\| {v}_q\|_{ L^1_t(L^p)}
            \right).
        \end{eqnarray*}
From Lemma \ref{vis2-L5.1}, we have for $ R_02^{ N_0-4}\leq 2^q\leq R_02^{ N_0}$,
    $$
    \|(\nabla E_{q},v_{q})\|_{L^\infty_t(L^p)}\leq C Y_s(t)\leq CY_{s,0},
    $$
    and
        \begin{eqnarray}
          &&\sum_{2^q> R_02^{ N_0}}2^{q(\frac{N}{p}-1)}\left(2^q\|  {c}_q\|_{L^\infty_t(L^p)\cap L^1_t(L^p)}
          +\| {v}_q\|_{L^\infty_t(L^p)}
          +2^{2q}\| {v}_q\|_{ L^1_t(L^p)}\right)\nonumber\\
            &\leq &CY_{s,0}+C \sum_{2^q> R_02^{ N_0}}2^{\frac{N}{p}-1}\left(
            \|( \Delta_q \nabla E_0, \Delta_q v_0)\|_{L^p}
            +  \|(\Delta_q\nabla L,\Delta_q G)\|_{L^1_t(L^p)}\right)\nonumber\\
            &&
              + C\int^t_0\|\nabla v\|_{L^\infty}\sum_{2^q> R_02^{ N_0}}2^{\frac{N}{p}-1}\|(\nabla E_{q},v_{q})\|_{L^p}d\tau\nonumber\\
              &&
            +C  \lambda_1 \sum_{2^q> R_02^{ N_0}}2^{\frac{N}{p}-1}\left(
            2^q\|  {E}_q\|_{L^\infty_t(L^p)\cap L^1_t(L^p)}
          +\| {v}_q\|_{L^\infty_t(L^p)}
          +2^{2q}\| {v}_q\|_{ L^1_t(L^p)}
            \right).\label{vis2-E5.53}
        \end{eqnarray}

           From   (\ref{vis2-E5.5})$_4$ and (\ref{vis2-E7.1}),
 we have
   \begin{eqnarray*}
   &&
      \|E^h \|_{\widetilde{L}^\infty_t(\dot{B}^{\frac{N}{p}}_{p,1})\cap
     \widetilde{L}^1_t(\dot{B}^{\frac{N}{p}}_{p,1})}\\
     &\leq& \|c^h \|_{\widetilde{L}^\infty_t(\dot{B}^{\frac{N}{p}}_{p,1})\cap
     \widetilde{L}^1_t(\dot{B}^{\frac{N}{p}}_{p,1})}
         +C \|(E^2)^h \|_{\widetilde{L}^\infty_t(\dot{B}^{\frac{N}{p}}_{p,1})\cap
     \widetilde{L}^1_t(\dot{B}^{\frac{N}{p}}_{p,1})}\\
      &\leq&\|c^h \|_{\widetilde{L}^\infty_t(\dot{B}^{\frac{N}{p}}_{p,1})\cap
     \widetilde{L}^1_t(\dot{B}^{\frac{N}{p}}_{p,1})}
     +C \|E \|_{\widetilde{L}^\infty_t(\dot{B}^{\frac{N}{p_2}}_{p_2,1})\cap
     \widetilde{L}^1_t(\dot{B}^{\frac{N}{p_2}}_{p_2,1})}
     \|E  \|_{L^\infty_t(\dot{B}^\frac{N}{p }_{p ,1})}    \\
      &\leq& \|c^h \|_{\widetilde{L}^\infty_t(\dot{B}^{\frac{N}{p}}_{p,1})\cap
     \widetilde{L}^1_t(\dot{B}^{\frac{N}{p}}_{p,1})}
   +C \lambda_1
   (\|E^l \|_{\widetilde{L}^\infty_t(\dot{B}^{s+1}_{2,r}) }
     +\|E^h  \|_{L^\infty_t(\dot{B}^\frac{N}{p }_{p ,1})}) .
  \end{eqnarray*}
When $\lambda_1$ is small enough, $C\lambda_1<\frac{1}{2}$, we have
 \begin{equation}
         \|E^h \|_{\widetilde{L}^\infty_t(\dot{B}^{\frac{N}{p}}_{p,1})\cap
     \widetilde{L}^1_t(\dot{B}^{\frac{N}{p}}_{p,1})} \leq  \|c^h \|_{\widetilde{L}^\infty_t(\dot{B}^{\frac{N}{p}}_{p,1})\cap
     \widetilde{L}^1_t(\dot{B}^{\frac{N}{p}}_{p,1})}
     +CY_{s,0}.\label{vis2-E5.54}
  \end{equation}

  From (\ref{vis2-E5.53})--(\ref{vis2-E5.54}) and
    $$
    \sum_{R_0< 2^q\leq  R_02^{ N_0}}  2^{q\frac{N}{p}}\|c_q\|_{L^\infty_t(L^p)\cap L^1_t(L^p)}
    \leq CY_s(t)\leq CY_{s,0},
    $$
    we have
  \begin{eqnarray*}
          &&\|E^h \|_{\widetilde{L}^\infty_t(\dot{B}^{\frac{N}{p}}_{p,1})\cap
     \widetilde{L}^1_t(\dot{B}^{\frac{N}{p}}_{p,1})}
          +\| {v}^h \|_{\widetilde{L}^\infty_t(\dot{B}^{\frac{N}{p}-1}_{p,1})\cap
     \widetilde{L}^1_t(\dot{B}^{\frac{N}{p}+1}_{p,1})} \nonumber\\
            &\leq &CY_{s,0}+CZ_{p,0}+C \|L^h\|_{\widetilde{L}^1_t(\dot{B}^{\frac{N}{p} }_{p,1})}
            +C \|G^h\|_{\widetilde{L}^1_t(\dot{B}^{\frac{N}{p}-1 }_{p,1})}\nonumber\\
            &&
              + C\int^t_0\|\nabla v\|_{L^\infty}\left(
              \|E^h \|_{\widetilde{L}^\infty_\tau(\dot{B}^{\frac{N}{p}}_{p,1}) }
          +\| {v}^h \|_{\widetilde{L}^\infty_\tau(\dot{B}^{\frac{N}{p}-1}_{p,1})}\right)d\tau\\
              &&
            +C  \lambda_1 \left(
\|E^h \|_{\widetilde{L}^\infty_t(\dot{B}^{\frac{N}{p}}_{p,1})\cap
     \widetilde{L}^1_t(\dot{B}^{\frac{N}{p}}_{p,1})}
          +\| {v}^h \|_{\widetilde{L}^\infty_t(\dot{B}^{\frac{N}{p}-1}_{p,1})\cap
     \widetilde{L}^1_t(\dot{B}^{\frac{N}{p}+1}_{p,1})}
                 \right).
        \end{eqnarray*}
        When $\lambda_1$ small enough, $C\lambda_1<\frac{1}{2}$, using the Gronwall's inequality, we get
          \begin{eqnarray}
          &&\|E^h \|_{\widetilde{L}^\infty_t(\dot{B}^{\frac{N}{p}}_{p,1})\cap
     \widetilde{L}^1_t(\dot{B}^{\frac{N}{p}}_{p,1})}
          +\| {v}^h \|_{\widetilde{L}^\infty_t(\dot{B}^{\frac{N}{p}-1}_{p,1})\cap
     \widetilde{L}^1_t(\dot{B}^{\frac{N}{p}+1}_{p,1})} \nonumber\\
            &\leq &CY_{s,0}+CZ_{p,0}+C \|L^h\|_{\widetilde{L}^1_t(\dot{B}^{\frac{N}{p} }_{p,1})}
            +C \|G^h\|_{\widetilde{L}^1_t(\dot{B}^{\frac{N}{p} -1}_{p,1})}.\label{vis2-E5.55}
        \end{eqnarray}
 Finally,  from Proposition \ref{vis2-P3.1-0}, we get
    \begin{eqnarray}
      \|L^h\|_{\widetilde{L}^1_t(\dot{B}^{\frac{N}{p} }_{p,1})}&\leq&\|\nabla v E\|_{\widetilde{L}^1_t(\dot{B}^{\frac{N}{p} }_{p,1})}
      +\|\partial_i(T'_Ev_i)\|_{\widetilde{L}^1_t(\dot{B}^{\frac{N}{p} }_{p,1})}\nonumber\\
        &\leq&C \|\nabla v\|_{\widetilde{L}^1_t(\dot{B}^{\frac{N}{p} }_{p,1})}\|E\|_{\widetilde{L}^\infty_t(\dot{B}^{\frac{N}{p_2} }_{p_2,1})}
        +C \|\nabla v\|_{\widetilde{L}^1_t(\dot{B}^{\frac{N}{p_2} }_{p_2,1})}\|E\|_{\widetilde{L}^\infty_t(\dot{B}^{\frac{N}{p} }_{p,1})}\nonumber\\
       &\leq& C\lambda_1 Z_p+CY_{s,0}.\label{vis2-E5.56}
    \end{eqnarray}
        Similarly, using   Proposition \ref{vis2-P3.1-0}, we get
	   \begin{equation}
  \|T'_{\nabla v_i}\cdot v\|_{\widetilde{L}^1_t(\dot{B}^{\frac{N}{p} -1}_{p,1})}= \|  T'_{ v }  v\|_{\widetilde{L}^1_t(\dot{B}^{\frac{N}{p}}_{p,1})}
  \leq C \|  v\|_{\widetilde{L}^1_t(\dot{B}^{\frac{N}{p_2}+1}_{p_2,1})}\|v\|_{\widetilde{L}^\infty_t(
  \dot{B}^{\frac{N}{p}-1}_{p,1})}\leq C\lambda_1 Z_p+CY_{s,0},
\end{equation}
and
    	   \begin{equation}
  \|E_{jk}\partial_j E_{ik}\|_{\widetilde{L}^1_t(\dot{B}^{\frac{N}{p} -1}_{p,1})}
 \leq \|E^2\|_{\widetilde{L}^1_t(\dot{B}^{\frac{N}{p}}_{p,1})}
  \leq C \|E\|_{\widetilde{L}^1_t(\dot{B}^\frac{N}{p_2}_{p_2,1})}\|E\|_{\widetilde{L}^\infty_t(
  \dot{B}^{\frac{N}{p}}_{p,1})}\leq C\lambda_1 Z_p+CY_{s,0}.
\end{equation}
From (\ref{vis2-E1.4}), we have
    \begin{eqnarray*}
      \Delta p&=&\partial_i(E_{jk}\partial_j E_{ik})-\partial_i(v_j\partial_j v_i) \\
      &=&\partial_i(E_{jk}\partial_j E_{ik})-\partial_i(T'_{\partial_j v_i}v_j)- \partial_j(T_{\partial_i v_j}v_i),
    \end{eqnarray*}
    then,
               \begin{eqnarray}
 \|G\|_{\widetilde{L}^1_t(\dot{B}^{\frac{N}{p} -1}_{p,1})}&\leq&\|\nabla p\|_{\widetilde{L}^1_t(\dot{B}^{\frac{N}{p} -1}_{p,1})}+
   \|E_{jk}\partial_j E_{ik}\|_{\widetilde{L}^1_t(\dot{B}^{\frac{N}{p} -1}_{p,1})}+  \|T'_{\nabla v}v\|_{\widetilde{L}^1_t(\dot{B}^{\frac{N}{p} -1}_{p,1})}\nonumber\\
    & \leq&   C\lambda_1 Z_p+CY_{s,0}.\label{vis2-E5.58}
\end{eqnarray}
From (\ref{vis2-E5.55})--(\ref{vis2-E5.56}) and (\ref{vis2-E5.58}), we have
    $$
    Z_p \leq CY_{s,0}+CZ_{p,0}+C\lambda_1 Z_p.
    $$
When $C\lambda_1\leq\frac{1}{2}$,  we can obtain  (\ref{vis2-E5.48-0}).
    \end{proof}
From a standard bootstrap argument, we obtain the following proposition.
\begin{prop}\label{vis2-P5.1}
Let $(v,E)$ be a solution to (\ref{vis2-E1.1}) which belongs to $V^s_{p_1,r}(T )$ with $s$,
$p_1$ and $r$ satisfying the condition of Theorem \ref{vis2-T1.2}. There exist two constants $\lambda$ and $C$
depending only on $N$, $p_2$ and $s$  such that if $X_{p_2,0}\leq\lambda$  then
    $$
    X_{p_i} (t)\leq C X_{p_i ,0}\ \textrm{for all}\ t\in[0, T ]
     \ \textrm{ and }\ i = 1, 2.
     $$
\end{prop}

\section{The proof of the global existence theorem}
So, let us proceed to the proof of global existence under the assumptions of Theorem
\ref{vis2-T1.2}.
To simplify the presentation, we assume throughout that $r < \infty$. The case $r =\infty$
follows from similar arguments. It is only a matter of replacing the strong topology
in $\dot{B}^{s}_{2,r}$ or $\dot{B}^{s+1}_{2,r}$ by weak topology.

\noindent\textbf{First step: smooth solutions.} We smooth out the  initial  data as follows,
    $$
   v_{0,n}=S_n v_0,\  E_{0,n}=S_n E_0.
    $$
Then, $(v_{0,n},E_{0,n})_{n\in\mathbf{N}}$ satisfy
    $$
   v_{0,n}\in \dot{B}^s_{2,r}\cap \dot{B}^{\frac{N}{2}-1}_{2,1},
     E_{0,n}\in \dot{B}^s_{2,r}\cap \dot{B}^\frac{N}{2}_{2,1},
    $$
and
    \begin{equation}
      u_{0,n} \rightarrow  u_0\ \textrm{ in }\dot{B}^s_{2,r},
      u_{0,n}^h\rightarrow  u_0^h\ \textrm{ in }  \dot{B}^{\frac{N}{p_1}-1}_{p_1,1},\label{vis2-E6.1}
    \end{equation}
    \begin{equation}
      E_{0,n}^l\rightarrow  E_0^l\ \textrm{ in }\dot{B}^s_{2,r},
      E_{0,n}^h\rightarrow  E_0^h\ \textrm{ in }\dot{B}^{s+1}_{2,r}\cap \dot{B}^\frac{N}{p_1}_{p_1,1}.\label{vis2-E6.2}
    \end{equation}
From Theorem \ref{vis2-T1.1}, for all $n\in\mathbf{N}$, we obtain a maximal solution $(v_n, E_n)$ over the time interval $[0,T^*_n)$ such that for all $T<T^*_n$, we have
    $$
    v_n\in \left(\widetilde{C}_T(\dot{B}^{\frac{N}{2}-1}_{2,1})\cap L^1_T(
    \dot{B}^{\frac{N}{2}+1}_{2,1})\right)^N
    \ \textrm{ and }E_n\in \left(\widetilde{C}_T(\dot{B}^{\frac{N}{2}}_{2,1})\right)^{N\times N}.
    $$
We claim that   $(v_n, E_n)$ is in $V^s_{p_1,r}(T)$ (and thus also in $V^s_{p_2,r}(T)$)  for all   $T<T^*_n$. Because $p_1\geq2$, we have
    $$
v_n^h\in \left( \widetilde{C}_T(\dot{B}^{\frac{N}{p_1}-1}_{p_1,1})\cap L^1_T(
    \dot{B}^{\frac{N}{p_1}+1}_{p_1,1})\right)^N
    \ \textrm{ and }    E_n^h\in \left(\widetilde{C}_T(\dot{B}^{\frac{N}{p_1}}_{p_1,1})
    \cap \widetilde{C}_T(\dot{B}^s_{2,r})\right)^{N\times N}.
    $$
Then, from the equation, we can show that the low frequencies of $E_n$    and of $v_n$ are on $\widetilde{C}_T(\dot{B}^s_{2,r})$ for all $T<T^*_n$.

\noindent\textbf{Second step: global existence and uniform bounds for $(v_n, E_n)$.} From Proposition \ref{vis2-P5.1}, there exists a constant $\lambda$ such that if $X_{p_2,0}^n \leq \lambda $, then
    \begin{equation}
      X_{p_i}^n(t)\leq C_1  X_{p_i,0}^n, \textrm{ for all }i=1,2\ t\in[0,T].\label{vis2-E6.3}
    \end{equation}
If we assume that   $X_{p_2,0}\leq \frac{\lambda}{2}$, then properties (\ref{vis2-E6.1})--(\ref{vis2-E6.2}) guarantees that the above smallness condition for the smoothed out initial data are satisfied for all large enough $n$.

Assume, by contradiction, that $T^*_n$ is finite. Then applying Proposition \ref{vis2-P5.1}, we get
    $$
    X^n_2(t)\leq CX^n_{2,0}\ \textrm{ for all} \ t\in[0,T^*_n).
    $$
So,  $v_n\in (\widetilde{L}^\infty_{T^*_n}(\dot{B}^{\frac{N}{2}-1}_{2,1}))^N$ and  $E_n\in (\widetilde{L}^\infty_{T^*_n}(\dot{B}^\frac{N}{2}_{2,1}))^{N\times N}$. Then, one can easily obtain that $(v_n,E_n)$ may be continued beyond $T^*_n$ into   a solution $(\tilde{v}_n,\tilde{E}_n)$ of (\ref{vis2-E1.4}) which coincides with $(v_n,E_n)$ on $[0,T^*_n)$ and such that, for some $T>T^*_n$,
$$
    v_n\in \left(\widetilde{C}_T(\dot{B}^{\frac{N}{2}-1}_{2,1})\cap L^1_T(
    \dot{B}^{\frac{N}{2}+1}_{2,1})\right)^N
    \ \textrm{ and }E_n\in\left( \widetilde{C}_T(\dot{B}^{\frac{N}{2}}_{2,1})\right)^{N\times N}.
    $$
And one can easily obtain that $(\tilde{v}_n,\tilde{E}_n)\in V^s_{p_1,r}(T)$. This stands in contradiction with the definition of $T^*_n$. Thus, $T^*_n=\infty$ and (\ref{vis2-E6.3}) holds true globally.

\noindent\textbf{Third step: passing to the limit.}
Using the   Aubin-Lions type argument \cite{Danchin2010}, we easily have that, up to extraction, $(v_n,E_n)$ converges to $(v,E)$, and $(v,E)\in V^s_{p,r}(\infty)$, $p=p_1$ or $p_2$, $(v,E)$ is the solution of (\ref{vis2-E1.4}).

Then, we can easily obtain that $(v,E)\in\mathcal{E}^{\frac{N}{p_1}}_T$. From the following local result in \cite{Qian}, we can obtain the uniqueness of the solution when $p_1\leq 2N$. This finishes the proof of Theorem \ref{vis2-T1.2}.

\begin{thm}[\cite{Qian}]
    Suppose that the
   initial data satisfies the incompressible constraints (\ref{vis2-E1.3-0}), $v_0\in
   ( \dot{B}_{p,1}^{\frac{N}{p}-1})^N$ and
 $E_0\in (\dot{B}_{p,1}^\frac{N}{p})^{N\times N}$, $p\in[1,2N]$.
    Then there exist $T>0$ and a  unique local solution for system (\ref{vis2-E1.4}) that satisfies
        $
  (v,E)\in  \mathcal{E}^\frac{N}{p}_T,
        $
where
$
    \mathcal{E}^s_T=\left(L^1([0,T]; \dot{B}_{p,1}^{s+1})\cap C([0,T];   \dot{B}_{p,1}^{s-1})
    \right)^N\times \left( C([0,T]; \dot{B}_{p,1}^s)
    \right)^{N\times N}.$
\end{thm}

\section*{Acknowledgements}   This work is supported partially by NSFC No.10871175, 10931007,
10901137,   Zhejiang Provincial Natural Science Foundation of China Z6100217,  and SRFDP
No.20090101120005.

\end{document}